\documentclass[12pt, twoside,article]{memoir}

\usepackage[english]{babel}
\usepackage{amsmath,amsfonts,amssymb,amsthm,
		     url,makeidx,bbm,enumitem,stmaryrd,
		      graphicx,tikz-cd,mathrsfs,marvosym,etoolbox,
		      pgfplots,extarrows,bussproofs, booktabs,
		     csquotes
			}
\usetikzlibrary{arrows.meta,quotes,babel,intersections,angles,matrix,calc,patterns,graphs,3d,positioning}
\usepackage[letterspace = 125]{microtype}

\newtheoremstyle{bfnote}
	{}
	{}
	{\itshape}
	{0pt}
	{\bfseries}
	{.}
	{.5em}
	{\thmname{#1} \thmnumber{#2}\thmnote{ (#3)}}

\theoremstyle{bfnote}
\newtheorem{thm}[paragraph]{Theorem}
\newtheorem{lem}[paragraph]{Lemma}

\newtheorem{cor}[paragraph]{Corollary}

\newtheoremstyle{bfdefinition}
	{}
	{}
	{}
	{0pt}
	{\bfseries}
	{.}
	{.5em}
	{\thmname{#1}\thmnumber{ #2}\thmnote{ (#3)}}

\theoremstyle{bfdefinition}
\newtheorem{defn}[paragraph]{Definition}

\newtheorem{exm}[paragraph]{Example}
\newtheorem{exms}[paragraph]{Examples}
\newtheorem{rem}[paragraph]{Remark}

\newtheoremstyle{bfremark}
	{}
	{}
	{}
	{0pt}
	{\bfseries}
	{:}
	{.5em}
	{\thmname{#1}\thmnumber{ #2}\thmnote{ (#3)}}

\theoremstyle{bfremark}

\counterwithout{equation}{chapter}

\newcommand{\itemref}[2]{\hyperref[#2]{(#1-{\ref*{#2}})}}

\makeatletter
\newcommand{\latex}{%
    L\kern -.36em{%
    \sbox \z@ T\vbox to\ht 0{\hbox {\check@mathfonts \fontsize \sf@size \z@ \math@fontsfalse \selectfont A}\vss }}\kern -.15em%
    T\kern-.1667em\lower.5ex\hbox{E}\kern-.125em X%
}
\newenvironment{prf}[1][\proofname]{\par
	\pushQED{\qed}%
	\normalfont\topsep6\p@\@plus6\p@\relax
	\trivlist
	\item\relax
		{\bfseries
		#1\@addpunct{:}}\hspace\labelsep\ignorespaces}
	{%
	\popQED\endtrivlist\@endpefalse
	}
\newcommand{\solutionname}{Solution}

\renewcommand{\l@paragraph}{\@dottedtocline{4}{1.5em}{2.3em}}

\makeatother

\def\@pnumwidth{2.2em}

\setsecnumdepth{paragraph}

\DeclareSymbolFont{stmry}{U}{stmry}{m}{n}

\tikzcdset{simto/.tip={Glyph[glyph math command=rightarrowtriangle]}}

\makeindex

\allowdisplaybreaks

\EnableBpAbbreviations

\DeclareMathSymbol\boxslash\mathbin{stmry}{"1B}

\DeclareMathOperator{\codom}{cod}
\DeclareMathOperator{\dHom}{dHom}
\DeclareMathOperator{\dom}{dom}
\DeclareMathOperator{\fHom}{fHom}

\DeclareMathOperator{\Fun}{Fun}

\newcommand{\Grothendieck}[2]{\sum_{#1}#2}

\DeclareMathOperator{\id}{id}

\DeclareMathOperator{\opp}{\!^{\mathrm{op}}}
\DeclareMathOperator{\quadd}{\quad\quad}

\DeclareMathOperator{\pr}{\mathsf{pr}}

\newcommand{\prone}[1]{\mathtt{prn}(#1)}

\newcommand{\rueckzug}[5][north west]{%
	\coordinate (a) at (\tikzcdmatrixname-#2)  ; %
	\coordinate (b) at (\tikzcdmatrixname-#2.east)  ; %
	\coordinate (b1) at ($(\tikzcdmatrixname-#2.south west)-( 4.5pt, 4.5pt)$) ; %
	\coordinate (b2) at ($(\tikzcdmatrixname-#2.south east)-(-4.5pt, 4.5pt)$) ; %
	\coordinate (b3) at ($(\tikzcdmatrixname-#2.north west)-( 4.5pt,-4.5pt)$) ; %
	\coordinate (b4) at ($(\tikzcdmatrixname-#2.north east)-(-4.5pt,-4.5pt)$) ; %
	\coordinate (c) at (\tikzcdmatrixname-#2.south)  ; %
	\coordinate (e) at (\tikzcdmatrixname-#3)  ; %
	\coordinate (f) at (\tikzcdmatrixname-#4)  ; %
	\path[name path = A] let
	\p3 = ($ (e) - (a)$) , %
	\p4 = ($ (f) - (a)$) , %
	\n3 = {veclen(\x3,\y3)} , %
	\n4 = {veclen(\x4,\y4)} , %
	\n5 = {5} , %
	\p5 = ($(\x3/\n3*\n5, \y3/\n3*\n5)$) ,%
	\p6 = ($(\x4/\n4*\n5, \y4/\n4*\n5)$) ,%
	\n6 = {.9*min(\n3,\n4)} , %
	\p7 = ($(\x3/\n3+\x4/\n4, \y3/\n3+\y4/\n4)$), %
	\p8 = ($(\x7*\n6 ,\y7*\n6)$) %
	in
	(a) -- ($(a) + (\p8)$)
	coordinate (A) at (\p5)
	coordinate (B) at (\p6);
	\path[save path = \pathB, name path = B] (b1) -- (b2) -- (b4) -- (b3) -- (b1);
	\draw[name intersections={of=A and B}]  (intersection-1) -- ($(intersection-1) - (A)$)
	(intersection-1) -- ($(intersection-1) - (B)$);
	\node[name intersections={of=A and B}, anchor = #1, inner sep = 0pt] (notiz) at  (intersection-1) {$\scriptscriptstyle\mathrm{#5}$};
}
\newcommand{\sasmap}[1]{(\hat{#1},\tilde{#1},#1)}
\newcommand{\spacer}[1][\kern0pt]{\,\text{-\lower2pt\hbox{$\scriptscriptstyle#1$}}\,}
\DeclareMathOperator{\spl}{spl}
\DeclareMathOperator{\str}{str}
\DeclareMathOperator{\To}{\Rightarrow}

\newcommand{\tv}[1]{\mathtt{tv}(#1)}

\newcommand{\bbone}{\text{\usefont{U}{bbold}{m}{n}1}}
\MakeRobust{\bbone}
\newcommand{\bbtwo}{\text{\usefont{U}{bbold}{m}{n}2}}
\MakeRobust{\bbtwo}
\newcommand{\bbthree}{\text{\usefont{U}{bbold}{m}{n}3}}
\MakeRobust{\bbthree}
\newcommand{\bbfour}{\text{\usefont{U}{bbold}{m}{n}4}}
\MakeRobust{\bbfour}
\newcommand{\bbfive}{\text{\usefont{U}{bbold}{m}{n}5}}
\MakeRobust{\bbfive}
\newcommand{\bbsix}{\text{\usefont{U}{bbold}{m}{n}6}}
\MakeRobust{\bbsix}
\newcommand{\bbseven}{\text{\usefont{U}{bbold}{m}{n}7}}
\MakeRobust{\bbseven}
\newcommand{\bbeight}{\text{\usefont{U}{bbold}{m}{n}8}}
\MakeRobust{\bbeight}
\newcommand{\bbnine}{\text{\usefont{U}{bbold}{m}{n}9}}
\MakeRobust{\bbnine}
\newcommand{\bbnull}{\text{\usefont{U}{bbold}{m}{n}0}}
\MakeRobust{\bbnull}

\DeclareMathOperator{\Cat}{\mathbf{Cat}}

\DeclareMathOperator{\CoAlg}{\mathsf{CoAlg}}

\DeclareMathOperator{\CompCat}{\mathsf{ComprC}}

\DeclareMathOperator{\depC}{\mathsf{depC}}
\DeclareMathOperator{\depSC}{\mathsf{(dep,\Sigma)C}}

\DeclareMathOperator{\Fib}{\mathsf{Fib}}

\DeclareMathOperator{\gCwF}{\mathsf{gCwF}}

\DeclareMathOperator{\HCompCat}{\mathsf{HComprC}}
\DeclareMathOperator{\OpFib}{\mathsf{OpFib}}

\DeclareMathOperator{\sas}{\mathsf{sas}}
\DeclareMathOperator{\Sets}{\mathsf{Sets}}
\DeclareMathOperator{\WCmd}{\mathsf{WCmd}}

\newcommand{\C}[1]{\mathscr{#1}}

\newlength{\boxi}
\newlength{\boxilaenge}
\newlength{\boxii}
\newlength{\abstractwdth}
\newlength{\keywordwdth}
\newlength{\authorwdth}
\newlength{\Infowdth}
\newcounter{Autorenzahl}
\makeatletter
\let\authors\@empty
\let\institutions\@empty
\renewcommand{\author}[4]{%
\ifx\@empty\authors
	\gdef\authors{\vtop{\hbox{\vtop{\hsize\boxilaenge\noindent#1~\textsc{#2}\textsuperscript{#4}\\\footnotesize\href{mailto:#3}{\texttt{#3}}\normalsize}}}\refstepcounter{Autorenzahl}}%
\else
	\g@addto@macro\authors{
	\vtop{
            \hbox{\kern10pt%
		    \vtop{\noindent#1~\textsc{#2}\textsuperscript{#4}}
	    \vtop{\footnotesize\href{mailto:#3}{\texttt{#3}}\normalsize}
	    \kern10pt%
		}
	}
	\refstepcounter{Autorenzahl}
	\ifnum\value{Autorenzahl}>2 %
		\g@addto@macro\authors{\\}
	\else
		\relax
	\fi
	}%
\fi}%
\newcommand{\institute}[2]{%
\ifx\@empty\institutions
	\gdef\institutions{\textsuperscript{#1}\!#2}%
\else
	\g@addto@macro\institutions{\textsuperscript{#1}\!#2}%
\fi
}%
\makeatother

\newcommand{\titlesetting}[5]{%
	\hrule%
	\kern25pt%
	\hbox to \textwidth{\vbox{#1}}%
	\kern25pt%
	\hrule%
	\hbox to \textwidth{%
		\vtop{%
			\kern 10pt%
            \hbox to .65\textwidth{%
				\kern10pt\vtop{\hsize\boxii \noindent\textls{Abstract}%
                \par\vskip-.5em\noindent                \rule{1.3\abstractwdth}{.4pt}%
                \par\noindent%
\small#2\normalsize\par\medskip\noindent\textls{Keywords}%
\par\vskip-.5em\noindent
                \rule{1.3\keywordwdth}{.4pt}%
                \par\noindent
\small#5}} \kern 10pt
		}%
        \vrule%
		\vtop{%
			\kern10pt%
            \hbox to \boxi{%
            \kern 10pt \vtop{\hsize\boxilaenge
            \noindent\textls{Authors}%
            \par\vskip-.5em\noindent\rule{1.3\authorwdth}{.4pt}%
            \par\smallskip\noindent%
            #3\normalsize}
            }
			\kern 10pt
			}
}
	\hrule%
    \hbox to \textwidth{%
		\vtop{%
			\kern 10pt\hbox to .96\textwidth{%
				\kern10pt\vtop{\hsize.96\textwidth \noindent\textls{Article Info}\par\vskip-.5em\noindent
                \rule{1.3\Infowdth}{.4pt}\par\noindent
\small#4}}\kern 10pt}}
    \hrule
}

\newcommand{\titlesetter}[3]{%
	\vbox{%
		\hrule
	\kern25pt%
	\hbox to \textwidth{\vbox{#1}}%
	\kern10pt
	\hbox to \textwidth{%
    \hfill
		\vtop{%
			\kern10pt%
			\hbox{%
				\hss
				\authors
               			\hss
		}
        \kern 15pt
        }
        \hfill
}
 \hbox to \textwidth{%
        \kern10pt \vtop{%
        \hbox{\noindent\small\institutions}\normalsize
        \kern 15pt}
	       }
}
	\hrule%
\hbox to \textwidth{%
		\vtop{%
			\kern 10pt\hbox to .96\textwidth{%
				\kern10pt\vtop{\hsize.96\textwidth \noindent\textls{Abstract}\par\vskip-.5em\noindent
                \rule{1.3\abstractwdth}{.4pt}\par\noindent
\small#2}}\kern 10pt
\hbox to .96\textwidth{%
	\kern10pt\vtop{\hsize.96\textwidth %
\noindent\textls{Keywords}\par\vskip-.5em\noindent
                \rule{1.3\keywordwdth}{.4pt}\par\noindent
\small#3}}\kern 10pt
}}
\hrule
}

\setstocksize{296mm}{210mm}		
\settrimmedsize{\stockheight}{\stockwidth}{*}
\settypeblocksize{250mm}{165mm}{*}	
\setlrmargins{*}{*}{1}
\setulmargins{*}{*}{1.22}
\checkandfixthelayout

\usepackage[style = alphabetic,backend=biber,backref, backrefstyle=none]{biblatex}
\setlist[enumerate]{leftmargin = 1.5em}
\setlist[itemize]{leftmargin = 1.5em}
\usepackage[hidelinks]{hyperref}
\hypersetup{
	pdfauthor 	={Luis Antonio Gambarte},
	pdftitle 	={2-dep,Sigma-categories are not generalised categories with families},
	pdfsubject	={We show that the generalised categories with families are not biequivalent to the (2-dep,Sigma)-categories of Petrakis},
	bookmarksnumbered
}

\usepackage{cleveref}

\author{Luis}{Gambarte}{gambarte@math.lmu.de}{1}
\institute{1}{Ludwig-Maximilians-Universität M\"unchen, Theresienstraße 39, D-80333 München}

\begin{document}

\setbeforeparaskip{.25em plus 1em minus .2em}
\pagestyle{plain}
\titlesetter{%
\centering\LARGE\bfseries (2-dep,$\Sigma$)-categories are not generalised categories with families%
}{%
	The notion of a generalised category with families, introduced by Coraglia and Emmenegger is one of the most general notions introduced to capture categorically the notion of dependent typing.
	It is shown that these generalised categories with families are biequivalent to comprehension categories. 
	We will show that the notion of a (2-dep,$\Sigma$)-category, an extension of the notion of a (dep,$\Sigma$)-category, introduced by Petrakis, is not biequivalent to generalised categories with families, but instead is equivalent to a direct generalisation of that notion.%
}{%
	Category theory, categories with families, dependent type theory, comprehension categories
}%
\normalsize
\kern22pt
\chapter{Introduction}

From the study of the semantics of Martin-Löf type theory and other type theories with dependent types a large variety of categorical structures has arisen.
These categorical structures roughly fall into one of two sorts, the first is oriented around indexed categories, the second around Grothendieck fibrations. 
The interplay between these two sorts rests on the correspondence ``indexed categories \( \mapsfrom \mapsto  \) fibrations'' established by Grothendieck.
Some of the structures falling into the first sort are categories with families~\cite{dybjerInternalTypeTheory1996}, categories with attributes~\cite{cartmellGeneralisedAlgebraicTheories1978} and fam-categories~\cite{petrakisCategoriesDependentArrows2023}. whereas comprehension categories~\cite{jacobsComprehensionCategoriesSemantics1993,jacobsCategoricalTypeTheory1991,jacobsCategoricalLogicType1999}, \( C \)-systems \cite{voevodskyCsystemModuleJfrelative2023} and many others fall into the second sort. 
Most of these structures are compared in~\cite{ahrensComparingSemanticFrameworks2025a}, where it is found that all of the structure examined therein are equivalent to some special type of comprehension category.
Generalised categories with families are introduced in~\cite{coragliaContextJudgementDeduction2024}, and in~\cite{coraglia2categoricalAnalysisContext2024} it is shown that these are equivalent to comprehension categories.

In~\cite{petrakisCategoriesDependentArrows2023} additionally dep-structures are introduced for fam-categories---which are generalised to 2-dep-structures for 2-fam-categories in~\cite{ehrhardt2depCategories2024}---which seems like stacking another layer of ``attributes'' or ``families'' on top of an existing one. 
However, this is not the case, as this paper aims to show. 
More specifically, we aim to show that (2-dep,$\Sigma$)-categories are equivalent to a generalisation of comprehension categories, which we will introduce, which generalise generalised categories with families in a straightforward way. 

To make matters simpler we do not work with fam-categories, but introduce dep-structures on indexed categories, which reduces the amount of new terminology, and is entirely equivalent, if one reduces oneself to the case that all involved classes are sets.
Table~\ref{table::terminology} gives an overview of the differences between the original terminology and the terminology in this paper.

\begin{table}[h]
	\caption{On the terminology in this paper}
	\centering
\begin{tabular}{ll}
	\toprule
	\emph{originial notion} & \emph{notion used in this paper}
	\\
	\cmidrule(lr){1-2}
	fam-categories & --
	\\
	2-fam-categories & indexed categories
	\\
	dep-categories & --
	\\
	2-dep-categories & dep-categories
	\\
	(2-fam,$\Sigma$)-categories & indexed categories with Sigma-objects
	\\
	(2-dep,$\Sigma$)-categories & (dep,$\Sigma$)-categories
	\\
	\bottomrule
\end{tabular}
\label{table::terminology}
\end{table}

Note that what we call a dep-category in this paper is equivalent to what is called a 2-dep-category in the literature. 
However, nothing is lost, as all the usual non ``2-'' notions can be recovered by considering indexed sets \( I \colon \C C\opp \to \Sets \).
This paper is structured as follows:
\begin{itemize}
	\item In Section~\ref{sec::indCat} we recall the notions of indexed categories and comprehension categories from the literature, and define what it means for a indexed category to have Sigma objects. (This corresponds to (2-fam,$\Sigma$)-categories in the original work.)
	\item In Section~\ref{sec::deparr} we introduce dep-structures on indexed categories and sas-towers, and we show that these two notions are 2-equivalent. 
	\item In Section~\ref{sec::depsig} we introduce (dep,$\Sigma$)-structures on indexed categories with Sigma-objects and higher comprehension categories and show that the two notions are 2-equivalent. 
	\item In Section~\ref{sec: Comparison} we recall the definition of a generalised category with families and compare it to the definition of a higher comprehension category and explain how the specialisation of a generalised category with families corresponds to the canonical (dep,$\Sigma$)-structure on an indexed category with Sigma objects, introduced in~\cite{petrakisCategoriesDependentArrows2023}.
\end{itemize}
For all notions from (2-)-category theory used herein without explanation we refer to \cite{barrCategoryTheoryComputing1991,johnson2DimensionalCategories2020}, for an introduction to the theory of fibrations we refer to~\cite{jacobsCategoricalLogicType1999}.

\par\bigskip\noindent
\textbf{Acknowledgements.}
We would like to thank Jacopo Emmenegger for the discussions at Luminy, which sparked our work on the topic of this paper, as well as Peter Lumsdaine for useful comments and suggestions on a first draft.

\chapter{Indexed categories and Sigma objects}
\label{sec::indCat}

We will briefly recall some facts about indexed categories and comprehension categories that we need for this paper. 
All results will be stated without proofs, but references to proofs will be given.

\begin{defn}[Indexed categories]
	A split indexed category is a functor \( I \colon \C C\opp \to \Cat \).
	A \emph{map} between split indexed categories \( I \to I' \) is a  pair \( (F,F_{\spacer}) \) of a functor \( F \colon \C C \to \C D \) and a natural transformation \( F_{\spacer} \colon I \Rightarrow J \circ F\opp \).
	In a diagram this could be visualised as
	\[ \begin{tikzcd}
		\C C\opp \ar[dr,"I"{name = U}] \ar[dd,"F\opp"'] 
		\ar[dd, Rightarrow,"F_{\spacer}"', shorten = 1.75em, shift left = 2em] 
		\\
		& \Cat.
		\\
		\C D\opp \ar[ur,"J"{name = V, swap}] 
	\end{tikzcd} \]
	A \emph{transformation between maps} \( (F,F_{\spacer}) \To (G,G_{\spacer}) \)  of split indexed categories is a 
	natural transformation \( \eta \colon F\opp \To G\opp \) such that \( (1_I \ast \eta) \circ F_{\spacer} =  G_{\spacer} \).
	\[ \begin{tikzcd}
		\C C\opp \ar[dr,"I"{name = U, near start}, to path = {(\tikztostart.east) -| (\tikztotarget.north) \tikztonodes}, rounded corners] 
		\ar[dd,"F\opp"{swap, name = V}, bend right = 30] 
		\ar[dd,"G\opp"{name = W}, bend left = 30]
		\\
		&[3.5em]\Cat.
		\\
		\C D\opp \ar[ur,"J"{swap, name = X, near start}, to path = {(\tikztostart.east) -| (\tikztotarget.south) \tikztonodes}, rounded corners]
		\ar[from = V, to = W, "\eta", Rightarrow,shorten = .3em]
		\ar[from = U, to = X,shorten = 2em, shift right = 1em,"F_{\spacer}"{swap, name = Y}, Rightarrow]
		\ar[from = U, to = X,shorten = 2em, shift left = 1em,"G_{\spacer}"{name = Z}, Rightarrow]
	\end{tikzcd} \]
	We denote the 2-category whose objects (0-cells) are split indexed categories, whose 1-cells are maps between  split indexed categories and whose 2-cells are transformations between this maps by \( \mathsf{IC} \).
\end{defn}

\begin{defn}[2-Category of fibrations]
	We will denote by \( \Fib \) the category with
	\begin{itemize}
		\item objects (0-cells) Grothendieck fibrations \( p \colon \C E \to \C B \),
		\item 1-cells maps between Grothendieck fibrations as defined in \cite{jacobsCategoricalLogicType1999}, that is functors \( \hat F \colon \C E \to \C E', F \colon \C E \to \C E' \) such that 
			\[ \begin{tikzcd}
				\C E \ar[d,"p"'] \ar[r,"\hat F"] & \C E' \ar[d,"p'"] 
				\\
				\C B \ar[r,"F"] & \C B'
			\end{tikzcd} \]
			 commutes and such that, if \( f \) in  \( \C E \) is \( p\)-cartesian for \( g \in \C B \), then \( \hat F(f) \) is \( p' \)-cartesian for \( F(g) \).
		 \item 2-cells are pairs of natural transformations \( \hat \eta \colon \hat F \To \hat G, \eta \colon F \To G \) such that 
\[ \begin{tikzcd}
	\C E \ar[d,"p"'] \ar[r,"\hat F"{name = U1}, bend left = 30] \ar[r,"\hat G"{swap, name = U2}, bend right = 30]  & \C E' \ar[d,"p'"] 
	\\[.4em]
	\C B \ar[r,"F"{name = V1}, bend left = 30] \ar[r,"G"{swap, name = V2}, bend right = 30] & \C B'
	\ar[from = U1, to = U2, "\hat\eta", Rightarrow, shorten = .2em]
	\ar[from = V1, to = V2, "\eta", Rightarrow, shorten = .2em]
\end{tikzcd} \]
commutes and such that \( \hat\eta_{e} \) is \( p' \)-cartesian for \( \eta_{p(e)} \).
	\end{itemize}
	The subcategory \( \Fib_{\spl} \) consists of all split fibrations and the 1- and 2-cells that respect those splittings on the nose.
\end{defn}

\begin{rem} 
	Our definition of the 2-cells in \( \Fib \) differs from the definition given in \cite{jacobsCategoricalLogicType1999}, where the requirement that \( \hat\eta_e \) is \( p' \)-cartesian for \( \eta_{p(e)} \) is dropped.
\end{rem}

A well-known fact is then the following:

\begin{thm} 
	There is a 2-equivalence \( \Fib_{\spl} \simeq \mathsf{IC} \).
\end{thm}

\begin{defn}[Comprehension categories -- \cite{jacobsComprehensionCategoriesSemantics1993}]

	A \emph{comprehension category} consists of a functor \( P \colon \C E \to \C B^\to\) for categories \( \C{E,B} \) (where \( \C B^\to \) is the arrow category) such that
	\begin{enumerate}[leftmargin = 3em]
	\renewcommand{\labelenumi}{(CC\textsubscript{\theenumi})}
	\item The functor \( p := \codom \circ P \colon \C E \to \C B \) is a Grothendieck fibration
		\[ \begin{tikzcd}
			\C E \ar[rr,"P"] \ar[dr,"p"'] 
			&[-1em]&[-1em] 
			\C B^\to \ar[dl,"\codom"] 
			\\
			& \C B
		\end{tikzcd} \]
		
	\item \( P  \) maps arrows cartesian with respect to \( p \) to arrows cartesian with respect to \( \codom \).
	\end{enumerate}
	A comprehension category is \emph{split} if the fibration \( p \) admits a splitting and \emph{discrete} if the fibration \( p \) is discrete.
\end{defn}

\begin{defn}[Pseudo-maps of comprehension categories]
	Let \( P \colon \C E \to \C B^\to \) and \( P' \colon \C E' \to \C B^{\prime\to} \) be comprehension categories. 
	A \emph{pseudo-map} from \( P \) to \( P' \) is given by a map \( (F,G) \colon p \to p' \) of Grothendieck fibrations together with a natural isomorphism \( \phi \colon  P' \circ F \Rightarrow G^\to \circ P \) (where \( G^\to\) is the lift of \( G \) to the arrow categories) lying over the identity transformation of \( F \). 
	This data can be visualised as follows:
	{\[ \begin{tikzcd}
			\C E \ar[drr,"P"] \ar[ddr,"p"] \ar[rrr,"F"] &[-1em]&[-1em]& 
		\C E' \ar[drr,"P'"] \ar[ddr,"p'"] 
		\\
								    && \C B^\to \ar[dl,"\codom"] \ar[rrr,"G^\to"' near start] &&[-1em]&[-1em] \C B^{\prime \to }
		\ar[dl,"\codom"]
		\\
		& \C B \ar[rrr,"G"] & && \C B'
		\ar[to  = 1-4,from = 2-3, equals,"\phi"',"\sim"{sloped},shorten = .5em]
	\end{tikzcd}\quad\hbox{and}\quad
\begin{tikzcd} 
	G\big(P_0(e)\big) \ar[dr,"P(e)"'] \ar[rr,"\phi_e"] &[-2em]&[-2em] P_0'\big(F(e)\big) \ar[dl,"F(P'(e))"]
	\\
							   & G(e).
\end{tikzcd}
\]}
A pseudo-morphism where \( \phi \) is the identity is called a \emph{strict} morphism.
\end{defn}
\begin{defn}[Transformations of pseudo-maps]
	\label{pseudomaptransformationdefn}
	A transformation \( (F,G,\eta) \Rightarrow (F'.G',\eta')  \) consists of a 2-map of fibrations \( (\tilde\alpha,\alpha) \colon (F,G) \to (F',G') \), such that for every \( A \in \C E \) we have
\[
	\eta'_{\lambda} \circ P'(\tilde\alpha_\lambda) = \alpha_{P(\lambda)} \circ \eta_\lambda
\]
for all \( \lambda \in \C E \).
\end{defn}

Sigma-objects were originally introduced in \cite{petrakisCategoriesDependentArrows2023} for fam-categories and then later extended to 2-fam-categories in \cite{ehrhardt2depCategories2024}.
Thus, in order to eschew size issues, we work with functors \( \C C\opp \to \Cat \) from the beginning and introduce Sigma-objects and dep-structures for those. 

\begin{defn}[Sigma-objects]
	Let \( I \colon \C C\opp \to \Cat \) be an indexed category. Then \( I \) has \emph{Sigma-objects}, if
	\begin{enumerate}[leftmargin = 2.5em]
	\renewcommand{\labelenumi}{($\Sigma$\theenumi)}
	\item for any \( c \in \C C \) and any \( \lambda \in I(c) \) there exists an object \( \Grothendieck{c}{\lambda} \) together with a \emph{first-projection-arrow} \( \pr_1^\lambda \colon \Grothendieck{c}{\lambda} \to c \) such that for any \( f \colon c' \to c \) there exists an arrow \( \Grothendieck{\lambda}{f} \colon \Grothendieck{c'}{I(f)(\lambda)} \to \Grothendieck{c}{\lambda} \) such that 
	\[ \begin{tikzcd}[execute at end picture={
			\rueckzug{1-1}{1-2}{2-1}{};
		}]
		\Grothendieck{c'}{I(f)(\lambda)} \ar[r,"\Grothendieck{\lambda}{f}"]  \ar[d,"\pr_1^{I(f)(\lambda)}"'] & \Grothendieck{c}{\lambda} \ar[d,"\pr_1^\lambda"] 
		\\
		c' \ar[r,"f"'] & c
	\end{tikzcd} \]
	is a pullback square.
	Furthermore the following conditions need to be met:
	\begin{enumerate}[leftmargin = 2.75em]
		\renewcommand{\labelenumii}{($\Sigma$\theenumi.\arabic{enumii})}
	\item\emph{Compatibility:} For any \( f\colon c'' \to c', g \colon c' \to c \) the equalities \( \Grothendieck{\lambda}{g \circ f} = \Grothendieck{\lambda}{g} \circ \Grothendieck{I(g)(\lambda)}{f}\) and \( \Grothendieck{\lambda}{1_c} = 1_{\Grothendieck{c}{\lambda}} \) must hold.
	\end{enumerate}
\item For any \( c \in \C C \), \( \lambda,\mu \in I(c) \) and \( \eta \colon \lambda \to \mu \), there exists an arrow \( \Grothendieck{\lambda,\mu}{\eta} \colon \Grothendieck{c}{\lambda} \to \Grothendieck{c}{\mu} \) making the triangle 
	\[ \begin{tikzcd}
		\Grothendieck{c}{\lambda} \ar[rr,"\Grothendieck{\lambda,\mu}{\eta}"] \ar[dr,"\pr_1^{\lambda}"'] 
		&& \Grothendieck{c}{\mu} \ar[dl,"\pr_1^\mu"]
		\\
		& c
	\end{tikzcd} \]
	commutative.
	Furthermore the following conditions must be met:
	\begin{enumerate}[leftmargin = 2.75em]
		\renewcommand{\labelenumii}{($\Sigma$\theenumi.\arabic{enumii})}
	\item for all \( \mu \colon \lambda \to \eta, \nu \colon \eta \to \gamma \) the equality
		\( \Grothendieck{\lambda,\gamma}{\nu \circ \mu} = \Grothendieck{\eta,\gamma}{\nu} \circ \Grothendieck{\lambda,\eta}{\mu} \) holds.
	\item for all \( \lambda \) the equality \( \Grothendieck{\lambda,\lambda}{1_\lambda} = 1_{\Grothendieck{c}{\lambda}} \) holds.
	\end{enumerate}
	\end{enumerate}
\end{defn}

\begin{exms}\leavevmode 
	\begin{enumerate}\label{exm::typecategories}
	\renewcommand{\labelenumi}{\theenumi.}
	\item Any type category (see~\cite{pittsCategoricalLogic2001}) is an indexed category with Sigma-objects by defining \( I \colon \C C\opp \to \Cat \) via \( I(c) := \mathit{Type}_\C C(c) \) and \( I(f)(\lambda) := f^*(\lambda) \).
		The Sigma object of \( c \) and \( \lambda \in \mathit{Type}_\C C(c) \) is defined via \( \Grothendieck{c}{\lambda} := c \rtimes \lambda \) and \( \Grothendieck{\lambda}{f} := f \rtimes \lambda \).
	\item\label{exm::ring1} Any ring \( R \) gives rise to an indexed category with Sigma-objects in the following way (see~\cite{ehrhardt2depCategories2024}): The underlying category \( \C C \) is the category given by the additive group \( (R,+) \), and the functor \(I \colon \C C\opp \to \Cat  \) maps \( \ast \) to the category whose objects are the elements of \( R \times R \), and morphisms \( g\colon (a,b) \to (c,d) \) are set functions---\emph{not} ring homomorphisms---\( \eta \colon R \to R \) such that 
	 \( \eta(a-b) = c-d \).
	 The functors \( I(r) \) for \( r \colon \ast \to \ast \) are then defined via \( I(f)(a,b) = (a+r,b+r) \).
	 
	 The Sigma-objects are given via \( \Grothendieck{\ast}{(a,b)} := \ast \) and \( \pr_1^{(a,b)} := a \cdot b	 \).
	 Furthermore \( \Grothendieck{(a,b)}{r} := r \cdot (1 + r + a + b) \) and \( \Grothendieck{(a,b),(c,d)}{\mu} := a\cdot b - c \cdot d \).
	\end{enumerate}
\end{exms}

The following lemma is then straightforward to prove.

\begin{lem}\label{lem: equivSigmaComprehension}
There exists a 2-equivalence 
\[
	\Sigma-\mathsf{IC} \simeq \CompCat_{\spl}^{\str},
\]
where \( \CompCat_{\spl}^{\str} \) is the 2-category of split comprehension categories, strict morphisms and strict transformations (see~\cite[Definitions~1--4]{ahrensComparingSemanticFrameworks2025a}).
\end{lem}

\chapter{Dependent arrow structures}
\label{sec::deparr}

\newcounter{depcounter}
\begin{defn}[Dependent arrows]\label{def: depstruc1}
	Let \( I \colon \C C\opp \to \Cat \) be a indexed category with Sigma-objects. 
	A \emph{dependent arrow structure} (or equivalently \emph{dep-structure}) on \( I \) consists of a set \( \dHom(c,\lambda) \), for every \( c \in \C C, \lambda \in I(c) \), whose elements are called \emph{dep-arrows},  together with
	\begin{enumerate}[leftmargin = 2.5em,itemindent = 1em]
	\renewcommand{\labelenumi}{(dep-\theenumi)}
	\item a precomposition operation that assigns to each \( f \colon c' \to c \) and each \( \Phi \in \dHom(c,\lambda) \) a \( \Phi \circ f \in \dHom\big(c',I(f)(\lambda)\big) \).
		This operation must fulfil the following conditions:
		\begin{enumerate}[leftmargin = 2.5em,itemindent=1em]
			\renewcommand{\labelenumii}{(dep-\theenumi.\arabic{enumii})}
		\item for all \( f \colon c'' \to c', g \colon c' \to c \) the equality \( (\Phi \circ g) \circ f = \Phi \circ (g \circ f) \) must hold.
		\item For all \( c \) the equality \( \Phi \circ 1_c = \Phi \) must hold.
		\end{enumerate}
	\item A postcomposition operation that assigns to each \( \eta \colon \lambda \to \mu \) and each \( \Phi \in \dHom(c,\lambda) \) a \( \eta \circ \Phi \in \dHom(c,\mu) \).
		This operation must fulfil the following conditions:
		\begin{enumerate}[leftmargin = 2.5em, itemindent = 1em]
			\renewcommand{\labelenumii}{(dep-\theenumi.\arabic{enumii})}
		\item for all \( \eta \colon \lambda \to \mu, \theta \colon \mu \to \nu \) the equality \( (\theta \circ \eta) \circ \Phi = \theta \circ (\eta \circ \Phi) \) must hold.
		\item For all \( \lambda \) the equality \( 1_\lambda \circ \Phi = \Phi \) must hold.
			\setcounter{depcounter}{\value{enumi}}
		\item For all \( \eta \colon \lambda \to \mu, f \colon c' \to c \) the equality
		 \( (\eta \circ \Phi) \circ f = I(f)(\eta) \circ (\Phi \circ f) \) must hold.
		\end{enumerate}
	\end{enumerate}
	We call an indexed category with a dependent arrow structure a \emph{dep-category}.
\end{defn}

\begin{exms} 
	\begin{enumerate}\label{exm::indexedcategories}
	\renewcommand{\labelenumi}{\theenumi.}
	\item Every indexed category can be endowed with a dep-structure by setting \( \dHom(c,\lambda) = \{\ast\} \) for every \( \lambda \), the compositions are then trivial.
	\item\label{exm::ring2} The indexed category stemming from a ring \( R \) can be endowed with the following dep-arrow structure (see~\cite{ehrhardt2depCategories2024}):
		for each \( (a,b) \), the set \( \dHom\big(\ast,(a,b)\big) \) is defined to be 
		\[
			\{ \mathfrak{p} \in \mathsf{Ideal}(R) ~\vert~ a-b \in \mathfrak{p}\}.
		\]
		The precomposition is defined via \( \mathfrak{p} \circ r = \mathfrak{p} \).
		The postcomposition is defined via \( \eta \circ \mathfrak{p} := \langle\eta(\mathfrak{p})\rangle \), the smallest ideal containing \( \eta(\mathfrak{p}) \).
	\end{enumerate}
\end{exms}
A standard result in fibred category theory is, that indexed categories are biequivalent to split fibrations. 
It remains to examine to which additional structure the dep-structures correspond. 
We will show that they correspond to adding a \emph{discrete ambifibration} on top of the split fibrations, which together constitute a sas-tower, a terminology we will introduce and examine in the remainder of this section.
Thus the resulting picture is the following:
\[
	F\colon \C C\opp \to \Cat\textcolor{red}{+ \text{ dep-structure }} \leftrightarrow 
	\begin{tikzcd} 
		|[color = red]| 
	\C D \ar[d,"Q", red]	\\
	\C E \ar[d,"p"] \\
				 \C C\end{tikzcd}
\]
\emph{Discrete ambifibration} are a special version of ambifibrations, a term coined by Sattler in~\cite{sattlerKocksFat$Delta$2017}.

\begin{defn}[(discrete) ambifibrations]
	Let \( \C C \) be a category with orthogonal factorisation system \( (\mathcal{L,R}) \). 
	A functor \( p \colon \C{B \to C} \) is an \emph{ambifibration} if both \( p_\mathcal{L},p_\mathcal{R} \), defined through the pullback squares
	\[
		\begin{tikzcd}[execute at end picture={
		\rueckzug{1-1}{1-2}{2-1}{};
	}] 
			\C E_\mathcal{L} \ar[d,"p_\mathcal{L}"'] \ar[r] 
			& \C B \ar[d,"p"]
			\\
			\mathcal{L} \ar[r] & \C C
		\end{tikzcd}
		\quad\text{ and }\quad
\begin{tikzcd}[execute at end picture={
		\rueckzug{1-1}{1-2}{2-1}{};
	}]
			\C E_\mathcal{R} \ar[d,"p_\mathcal{R}"'] \ar[r] 
			& \C B \ar[d,"p"]
			\\
			\mathcal{R} \ar[r] & \C C
		\end{tikzcd}
	\]
	respectively, are opfibrations and fibrations, respectively.
	If they are a discrete opfibration and a fibration, respectively, we call \( I \) a \emph{discrete ambifibration}.
\end{defn}
We first show that the choice of pullback is not relevant for this definition, that is if we are given two pullbacks 
\begin{equation}
	\begin{tikzcd}[execute at end picture={
			\rueckzug{1-1}{1-2}{2-1}{};
		}]
		\C E_\mathcal{L}  \ar[r] \ar[d,"p_\mathcal{L}"'] & \C E \ar[d,"p"] 
		\\
		\mathcal{L} \ar[r] & \C B
	\end{tikzcd}
	\text{ and }
	\begin{tikzcd}[execute at end picture={
			\rueckzug{1-1}{1-2}{2-1}{};
		}]
		\C E_\mathcal{L}'  \ar[r] \ar[d,"p_\mathcal{L}'"'] & \C E \ar[d,"p"]
		\\
		\mathcal{L} \ar[r] & \C B
	\end{tikzcd}
	\label{AmbiFibPullEq::1}
\end{equation}
(similarly for \( p_\mathcal{R} \) in place of \( p_\mathcal{L} \)) then one being an opfibration implies the other being an opfibration, and vice versa. Explicitly:

\begin{lem} 
	In the situation of \eqref{AmbiFibPullEq::1} the following are equivalent:
	\begin{enumerate}
	\renewcommand{\labelenumi}{\theenumi.}
	\item\label{AmbiFibLem1::1} \( p_\mathcal{L} \) is a (split) opfibration.
	\item\label{AmbiFibLem1::2} \( p_\mathcal{L}' \) is a (split) opfibration.
	\end{enumerate}
\end{lem}
\begin{prf}
	We can build the pasted diagram
	\[ \begin{tikzcd}[execute at end picture={
			\rueckzug{1-1}{1-2}{2-1}{};
		}]
		\C E_\mathcal{L} \ar[r,"i","\cong"'] \ar[d,"p_\mathcal{L}"'] & \C E_\mathcal{L}' \ar[d,"p_\mathcal{L}'"] \ar[r] & \C E \ar[d,"p"]
		\\
		\mathcal{L} \ar[r,"\id_\mathcal{L}","\cong"'] & \mathcal{L} \ar[r] & \C B
	\end{tikzcd} \]
	 where the outher and the right square are pullbacks and thus the left is as well.
	 As opfibrations are stable under pullbacks \( p_\mathcal{L} \) has to be an opfibration if \( p_\mathcal{L}' \) is. 
	 For the reverse direction we observe that both \( i  \) and \( \id_\mathcal{L}  \) are isomorphisms, so we can exchange \( \C E_\mathcal{L} \) and \( \C E_\mathcal{L}' \).
\end{prf}

Naturally this extends to a dual result for \( p_\mathcal{R}\) which we will not state explicitly here.
These two results allow us to always check the ambifibration property through the categories \( \C E_\mathcal{L},\C E_\mathcal{R} \), explicitly defined as follows:
\begin{itemize}
	\item They both share the same objects as \( \C E \).
	\item The arrows in \( \C E_\mathcal{L} \) are those arrows \( f  \) in \( \C E \) such that \( p(f) \in \mathcal{L} \), similarly those in \( \C E_\mathcal{R} \) are those projected to an arrow in \( \mathcal{R} \) under \( p \).
\end{itemize}
\begin{defn}[Sas-towers]\label{sastowerdef}
	\index{Sas-tower|hyperit}
	\index{Tower!sas-\rule{1em}{.4pt}|see{Sas-tower}}
	We define a \emph{sas-tower} to consist of two functors \( Q \colon \C G \to \C E, p \colon \C E \to \C B \) such that 
	\begin{enumerate}[itemindent = 2em]
		\renewcommand{\labelenumi}{(sas-{\theenumi})}
		\newcounter{sadcounter}
	\item $p$ is a split Grothendieck fibration,
	\item\label{sastowerdef::item2} \( Q \) is a discrete ambifibration with respect to the orthogonal factorisation system \( (\mathcal{V,C}) \) of morphisms vertical/cartesian with respect to \( p \) and
	\item\label{sastowerdef::item3}  \(q:=  p \circ Q \) is a split Grothendieck fibration such that for every \( f \colon c \to c' \) in \( \C B \), every \( e \in \C E \) with \( p(e) = c' \) and every \( \Phi \in \C G \) with \( Q(\Phi) = e \) we have
		\[
			\overline{f}^{q}(\Phi) = \overline{ \overline{f}^p(e)}^{Q}(\Phi).
		\]
	\item\label{sastowerdef::item4}For any \( f \colon c \to c' \) in \( \C B \), \( g \colon e \to e' \) in \( \C E \) with \( p(g) = 1_{c'}  \) and any \( \Phi \in \C G \) with \( Q(\Phi) = e \) we have 
		\[
			{f}^{*q}\big({g}_{*Q}(\Phi)) = (g \bullet f)_{*Q}({f}^{*q}(\Phi)),
		\]
		where \( g \bullet f \) is defined through the universal property of the cartesian lifts of \( f \) as in 
		\[ \begin{tikzcd}
			f^*(e) \ar[d,dotted,"g \bullet f"'] \ar[r,"\overline{f}^p(e)"] & e \ar[d,"g"] 
			\\
			f^*(e') \ar[r,"\overline{f}^p(e')"] & e.
		\end{tikzcd} \]
		\setcounter{sadcounter}{\value{enumi}}
	\end{enumerate}
\end{defn}

\begin{lem}\label{2depAmbiLem::1}
	We can associate to any dep-category \( I \colon \C C\opp \to \Cat \) a sas-tower in the following way:
	\begin{itemize}
		\item The categories \( \C{E,B}\) and the split fibration \( p \colon \C E \to \C B \) are defined as in the translation between indexed categories and split fibrations. 
		\item The category \( \C G \) has
			\begin{itemize}
				\item as objects dep-arrows \( \Phi \) and
				\item a morphism \( (f,\eta) \colon \Phi \to \Psi \) consists of an arrow \( (f,\eta) \colon (c,\lambda) \to (c',\mu) \) in \( \C E \) such that  \( \Phi \in \dHom(c,\lambda), \Psi \in \dHom(c',\mu) \) and 
					\[
						\eta \circ \Phi = \Psi\circ f.
					\]
				\end{itemize} 
		\item The functor \( Q \colon \C G \to \C E \) takes each \( \Phi \) to the pair \( (c,\lambda) \) such that \( \Phi \in \dHom(c,\lambda) \) and \( (f,\eta) \colon \Phi \to \Psi \) to \( (f,\eta) \).
	\end{itemize}
\end{lem}

\begin{prf} 
	To see that \( q \colon \C G \to \C B \) is a split fibration, we observe that defining \( \overline{f}(\Phi):= (f,1_\lambda) \colon \Phi \circ f \to \Phi \) yields a splitting---the proof that \( (f,1_\lambda) \) is cartesian for \( f \) and \( \Phi \) is analogous to the proof that it is cartesian for \( f \) and \( \lambda \)---and by definition of a  dep-category  this cleavage fulfils functoriality, thus \( q \) is a split fibration.
We know that the orthogonal factorisation system on \( \C E \) in vertical and cartesian morphisms looks as follows:
	\begin{align*} 
		\mathcal{V} &= \{ (f,\eta) ~\vert~ f \text{  is an isomorphism}\},
		\\
		\mathcal{C} &= \{ (f,\eta) ~\vert~ \eta \text{ is an isomorphism}\}. 
	\end{align*}
	Thus we need to exhibit the discrete lifting (oplifting) property for \( Q_{\mathcal{C}} \) (respective \( Q_{\mathcal{V}} \)).
	Given an arrow in \( \mathcal{V} \), that is \( (f,\eta) \colon (c,\lambda) \to (c',\mu) \)--- with \( f  \) invertible---and \( \Phi \in \dHom(\lambda)\) the lift of \( (f,\eta) \) along \( \Phi \) is defined to be \( (f,\eta) \colon \Phi \to (\eta \circ \Phi)\circ f^{-1} \).
	This is an arrow in \( \C G \) as
	\[
		\big((\eta \circ \Phi)\circ f^{-1}\big)\circ f = (\eta \circ \Phi)\circ (f^{-1} \circ f) = \eta \circ \Phi.
	\]
	It is the necessarily unique lift of \( (f,\eta) \) with respect to \( Q_\mathcal{V} \).

	Similarly, given an arrow in \( \mathcal{C} \), that is \( (f,\eta) \colon (c,\lambda) \to (c',\mu) \)---with \( \eta \) invertible---and \( \Phi \in \dHom(c',\mu) \) the lift of \( (f,\eta)  \) along \( \Phi \) is defined to be \( (f,\eta) \colon \eta^{-1} \circ( \Phi\circ f) \to \Phi \). 
	This is a well-defined arrow as
	\[
	\eta \circ \big( \eta^{-1} \circ (\Phi \circ f)\big) = (\eta \circ \eta^{-1}) \circ (\Phi \circ f) = \Phi\circ f.
	\]
	It is also the necessarily unique lift of \( (f,\eta) \) with respect to \( Q_\mathcal{C} \).
	
	It remains to check the equalities imposed upon the splitting. 
	We compute that
	\[
		\overline{f}^{q}(\Phi) = \big((f,1_\lambda)\colon\Phi \circ f \to \Phi\big) = \overline{(f,1_\lambda)}^Q(\Phi) = \overline{\overline{f}^p(e)}^Q(\Phi).
	\]
	Similarly,  for \( f \colon c' \to c, \eta \colon \lambda \to \mu \) with \( p(\eta) = 1_c \) and \( \Phi \) with \( q(\Phi) = c' \) we calculate 
	\begin{align*}
		& {f}^{*q}\big( (1_c,\eta)_{*Q}(\Phi)\big) =  {f}^{*q}\big(\eta \circ \Phi)\big) =  (\eta \circ \Phi)\circ f \text{ and }
		\\
		&{(1_c,\eta) \bullet f}_{*Q}\big(f^{*q}(\Phi\big) = {(1_{c'},\eta\bullet f)}_{*Q}(\Phi \circ f) = (\eta \bullet f) \circ (\Phi\circ f).
	\end{align*}
	As \( (\eta\circ \Phi)\circ f = I(f)(\eta) \circ (\Phi \circ f) \) in any indexed category with a dep-structure this proves the required equality.
\end{prf}

Conversely, every sas-tower gives rise to an indexed category with a dep-structure.

\begin{lem}\label{sasto2dep} 
	Let \( (Q \colon \C G \to \C E, p \colon \C E \to \C B)\) be a sas-tower. 
	Then we obtain an associated indexed category \( I \colon \C C\opp \to \Cat \) with dep-structure defined by:	\begin{itemize}
		\item the dependent arrows over \( e \in \C E \) are defined as \( Q^{-1}(e) \).
		\item 
	The compositions of \( f \colon c' \to c, e \in \fHom(c), g\colon e' \to e \in \fHom(c), \Phi \in \dHom(e) \) are defined as 
	\begin{align*}	
		\Phi \circ f &:= f^*(\Phi), &&\text{the lift with respect to }q.
		\\
		g \circ \Phi &:= g_*(\Phi), &&\text{the lift with respect to }Q_\mathcal{V}.
	\end{align*}
\end{itemize}
\end{lem}

\begin{prf} 
	We already know that we can obtain an indexed category from the split fibration \( p \colon \C E \to \C B \).
	To see that it has a dep-structure, we first observe that the functoriality of composition \( \Phi \circ f \) and \( g \circ \Phi \) is guaranteed, as both are lifts with respect to discrete fibrations.
	To check \( (g \circ \Phi) \circ f = (g \circ f) \circ (\Phi \circ f) \), we simply note that both are mapped to the same element of \( \C E \) by \( Q \), and thus must be the same.
\end{prf}

The notion of an arrow between sas-towers will be the expected one, but before we state it, we need a few technical lemmata.

\begin{lem} 
	Let \( \C C \) be a category and let maps as in the diagram
	\[ \begin{tikzcd}[execute at end picture={
			\rueckzug{2-2}{2-4}{4-2}{};
		}]
		a \ar[rr,"f_1"] \ar[dd,"f_2"]  &&  \cdot\ar[dd,"g_1"{near end}] \ar[dr,"g_2"]
		\\
							&c \ar[rr,"h_1",{near start}] \ar[dd,"h_2"{near start}] &&  \cdot\ar[dd,"k"] 
		\\
		 \cdot\ar[rr,"l_1"'{near end}] \ar[dr,"l_2"] & &  \cdot\ar[dr,"m"] 
		\\
				     & \cdot\ar[rr,"n"] && \cdot 
	\end{tikzcd} \]
	be given such that all faces commute. 
	Then we obtain a map \(f_3 \colon a \to c \), making the remaining faces commutative.
\end{lem}

\begin{prf} 
	This follows as 
	\[
		k \circ g_2 \circ f_1 =  m \circ g_1 \circ f_1 = m \circ l_1 \circ f_2 = n \circ l_2 \circ f_2, 
	\]
	thus we can use the universal property of the pullback at \( c \).
\end{prf}

\begin{cor}\label{Ambifib2depCor::1}
	Given two ambifibrations \( Q \colon \C E \to \C B, Q' \colon \C E' \to \C B' \) and two functors \( \hat F \colon \C E \to \C E', F \colon \C B \to \C B' \) such that \( F \) maps arrows in \( \mathcal{L} \) to arrows in \( \mathcal{L}' \) and arrows in \( \mathcal{R} \) to \( \mathcal{R}' \), and
	\[ \begin{tikzcd}
		\C E \ar[r,"\hat F"] \ar[d,"Q"'] & \C E' \ar[d,"Q'"] 
		\\
		\C B \ar[r,"F"] & \C B'
	\end{tikzcd} \]
	commutes, we obtain \( \hat F_\mathcal{L} \colon \C E_\mathcal{L} \to \C E_{\mathcal{L}'} \) and \( \hat F_\mathcal{R} \colon \C E_\mathcal{R} \to \C E'_{\mathcal{R}'} \) such that 
	\[ \begin{tikzcd}
		\C E_\mathcal{L} \ar[r,"\hat F_\mathcal{L}"] \ar[d,"Q_\mathcal{L}"'] 
		& \C E'_{\mathcal{L}'} \ar[d,"Q'_{\mathcal{L}'}"] 
		\\
		\mathcal{L} \ar[r,"F_\mathcal{L}"] & \mathcal{L'}
	\end{tikzcd} \text{ and }
	\begin{tikzcd}
		\C E_\mathcal{R}\ar[r,"\hat F_\mathcal{R}"] \ar[d,"Q_\mathcal{R}"'] 
		& \C E'_{\mathcal{R}'} \ar[d,"Q'_{\mathcal{R}'}"] 
		\\
		\mathcal{R} \ar[r,"F_\mathcal{R}"] & \mathcal{R'}
	\end{tikzcd} 
\]
commute (where \( F_\mathcal{L},F_\mathcal{R} \) are obtained through restriction).	
\end{cor}

\begin{prf} 
	We obtain diagrams
	\[ \begin{tikzcd}[execute at end picture={
			\rueckzug{1-1}{1-3}{3-1}{};
			\rueckzug{2-2}{2-4}{4-2}{};
		}]
		\C E_ \mathcal{L} \ar[rr]  \ar[dd,"Q_\mathcal{L}"] 
		&& \C E \ar[dr,"\hat F"] \ar[dd,"Q"{near end}] 
		\\
		& \C E'_{\mathcal{L}'} \ar[rr] \ar[dd,"Q'_{\mathcal{L}'}"{near end}] 
		&& \C E' \ar[dd,"Q'"] 
		\\
		\mathcal{L} \ar[rr] \ar[dr,"F_\mathcal{L}"'] 
		&& \C B \ar[dr,"F"] 
		\\
		& \mathcal{L}' \ar[rr] && \C B'
	\end{tikzcd}\text{ and }\begin{tikzcd}[execute at end picture={
			\rueckzug{1-1}{1-3}{3-1}{};
			\rueckzug{2-2}{2-4}{4-2}{};
		}]
		\C E_ \mathcal{R} \ar[rr]  \ar[dd,"Q_\mathcal{R}"] 
		&& \C E \ar[dr,"\hat F"] \ar[dd,"Q"{near end}] 
		\\
		& \C E'_{\mathcal{R}'} \ar[rr] \ar[dd,"Q'_{\mathcal{R}'}"{near end}] 
		&& \C E' \ar[dd,"Q'"]
		\\
		\mathcal{R} \ar[rr] \ar[dr,"F_\mathcal{R}"'] 
		&& \C B \ar[dr,"F"] 
		\\
		& \mathcal{R}' \ar[rr] && \C B'
	\end{tikzcd}\ \]
	so, by the preceding lemma, we get the desired functors \( \hat F_\mathcal{L}, \hat F_\mathcal{R} \).
\end{prf}

\begin{defn}[Maps between sas-towers]\label{sasmapsdef}
	\index{Sad-tower!map|hyperit}
	Let \((Q \colon \C G \to \C E, p \colon \C E \to \C B), (Q' \colon \C G' \to \C E', p' \colon \C E' \to \C B')\) be two sas-towers. 
	A \emph{map} \( (Q,p) \to (Q',p') \) is a triple \( (\hat F,\tilde F,F) \) of functors
	\[
		\hat F \colon \C G \to \C G', \tilde F \colon \C E' \to \C E, F \colon \C B \to \C B'
	\]
	such that 
	\begin{enumerate}[itemindent = 2em]
		\renewcommand{\labelenumi}{(sas-{\theenumi})}
		\setcounter{enumi}{\value{sadcounter}}
	\item \( (\hat F,F) \) is a map of split fibrations,
	\item \( (\tilde F,F) \) is a map of split fibrations  and
	\item Both \( (\hat F_\mathcal{V},F), (\hat F_\mathcal{C},F) \) are maps of discrete fibrations, where \( \hat F_\mathcal{V},\hat F_\mathcal{C} \) are obtained as in the preceding corollary.
		\setcounter{sadcounter}{\value{enumi}}
	\end{enumerate}
	
\end{defn}

Before we can prove that any arrow between sas-towers corresponds to an arrow between dep-categories, we need to define  arrows between dep-categories. 
For this we first clarify what our arrows between indexed categories are.

\begin{defn} 
	Let \( I\colon \C C\opp \to \Cat, J \colon \C D\opp \to \Cat\) be dep-categories. 
	A dep-functor \( I \to J \) consists of a triple \( (F,F_{\spacer}, F^{\spacer})\) such that \( (F,F_{\spacer}) \) is a map between indexed categories and \( F^{\spacer} \) is a family of maps \( F^\lambda \colon \dHom(c,\lambda) \to \dHom\big(F(c),F_c(\lambda)\big) \) for every \( \lambda \), subject to the following conditions
	\begin{enumerate}[leftmargin = 2.5em, itemindent = 1em]
	\renewcommand{\labelenumi}{(dep-\theenumi)}
	\setcounter{enumi}{\value{depcounter}}
	\item\label{depeq::1} For all \( \lambda, f \colon c' \to c \) and all \( \Phi \in \dHom(c,\lambda) \), the equality \( F^{F(f)(\lambda)}(\Phi \circ f)  = F^\lambda(\Phi) \circ F(f)\) holds.
	\item\label{depeq::2} For all \( \lambda,\mu \) and \( \eta \colon \lambda \to \mu \) and all \( \Phi \in \dHom(c,\lambda) \), the equality \( F^\mu(\eta \circ \Phi) = F_c(\eta) \circ F^\lambda(\Phi)  \) holds.
		\setcounter{depcounter}{\value{enumi}}
	\end{enumerate}
\end{defn}

\begin{lem}\label{2depfuntomap}
	Let \(I\colon \C C\opp \to \Cat,I' \colon  \C C'\opp \to \Cat\) be two dep-categories and
	\(( Q_I, p_I)\), \( (Q_{I'}, p_{I'}) \)  the corresponding sas-towers obtained as in Lemma \ref{2depAmbiLem::1}. 
	Then we obtain for every dep-functor \( (F,F_{\spacer},F^{\spacer}) \colon \C C \to \C C' \) a map \( (\hat F, \tilde F, F) \colon ( Q_I,p_I) \to (Q_I,p_I)  \) of sas-towers.
\end{lem}

\begin{prf} 
	The map is defined as follows: The functors \( (\tilde F, F)  \) are defined as the map of split fibrations obtained from the underlying maps of indexed categories.
	The functor \( \hat F \colon \C G \to \C G' \) is defined by
	\[
		\hat F(\Phi) := F_\lambda(\Phi) \text{ and }\hat F(f,\eta) = (f,\eta).
	\]
	It follows from Corollary \ref{Ambifib2depCor::1} that both \( \hat F_\mathcal{V},\hat F_\mathcal{C}  \) are maps of discrete  fibrations---that is they commute with \( Q \) restricted to either \( \mathcal{C} \) or \( \mathcal{V} \)---and \( (\hat F,F) \) is also a map of split fibrations, as
	\[
		q_{I'}\big(\hat F(\Phi)\big) = F(c) = F\big(q_I(\Phi)\big),\quad q_{I'}\big(\hat F(f,\eta)\big) = F(f) = F\big(q_I(f,\eta)\big).
	\]
	and for \( f \colon c' \to c, \lambda \in \fHom(c) \) and \( \Phi \in \dHom(\lambda) \) we have 
	\[
		\overline{F(f)}\big(\tilde{F}(\Phi)\big) = \big(F(f),1_{F_c(\lambda)}\big) = \tilde F(f,1_\lambda) = F\big(\overline{f}(\Phi)\big).
	\]
	Hence \( (\hat F, \tilde F, F) \) is a map of sas-towers.
\end{prf}

Conversely, every map of sas-towers yields a dep-functor between the corresponding dep-categories.

\begin{lem}\label{sasmapto2depfun}
	Let \( (Q'\colon \C G' \to \C E', p' \colon \C E' \to \C B') \) and \((Q \colon \C G \to \C E, p \colon \C E \to \C B ) \)
	be sas-towers and \( I \colon \C C\opp \to \Cat,I'\colon \C C'\opp \to \Cat \) be the associated dep-categories from Lemma \ref{sasto2dep}. 
	Then any map \( (\hat F, \tilde F, F) \colon (Q',p') \to (Q,p) \) of sas-towers translates to a dep functor \( (F,F_{\spacer}, F^{\spacer}) \colon I \to I' \) where 
	\[
		F_c(e) := \tilde F(e) \text{ and } F^e(\Phi) := \hat F(\Phi).
	\]
\end{lem}

\begin{prf} 
	As in the case for indexed categories \( I \), is already well-defined.
	To see that \(F^e \) are well-defined, we observe that for all \(e \in \fHom(c)  \), that is \( p(e) = c \), and all \( \Phi \in \dHom(e) \), that is \( Q(\Phi) = e \), 
	\[
		Q\big(F^e(\Phi)\big) = Q\big(\hat F(e)\big) = \tilde F(Q'(\Phi)) = F_{F(c)}\big(Q'(\Phi)\big)	
	\]
	as required.
	To see the equalities \hyperref[depeq::1]{(dep-\ref*{depeq::1})} and \hyperref[depeq::2]{(dep-\ref*{depeq::2})}  we compute
	\begin{align*} 
	&	F^e(\Phi \circ f) = \hat F(\Phi \circ f) = \hat F\big(\overline{f}\Phi)\big) = \overline{F(g)}\big( \hat F(\Phi)\big) \text{ for }(f \colon c' \to c) \in \C B.
	\\
	&	F^e(g \circ \Phi) = \hat F(g \circ \Phi) = \hat F\big(\underline{g}(\Phi)\big) = \underline{\tilde F(g)}\big( \hat F(\Phi)\big) \text{ for }(g \colon e' \to e) \in \mathcal{V},
	\end{align*}
	thus \( F \) is a dep-functor.
\end{prf}

It is immediate from the above definition that this assignment of sas-maps to dep-functors fulfils functoriality.

\begin{defn}[Transformations of maps of sas-towers]\label{sastrafodef}
	Let \( (Q,p),(Q',p') \) be sas-towers and \( (\hat F, \tilde F, F),(\hat G, \tilde G, G) \colon (Q,p) \to (Q',p') \) be maps of sas-towers.
	We define a \emph{transformation of sas-maps} \( (\hat F,\tilde F, F) \to (\hat G, \tilde G,G) \) to be a triple \( (\hat \eta, \tilde \eta,\eta) \) where
	\[
		\hat \eta \colon \hat F \To \hat G,\quad \tilde \eta \colon \tilde F \To \tilde G,\quad \eta \colon F \To G
	\]
	are natural transformations such that the following conditions are met:
	\begin{enumerate}[itemindent = 2em]
		\renewcommand{\labelenumi}{(sas-{\theenumi})}
		\setcounter{enumi}{\value{sadcounter}}
	\item Both \( Q' \ast \hat \eta = \tilde \eta \ast Q \) and \( p' \ast \tilde \eta = \eta \ast p \), that is the diagram
		\[ \begin{tikzcd}
			\C G \ar[r,"Q"'] \ar[d,shift left = .5ex, bend left = 20,"\hat F"{name = U1}]
			\ar[d,shift right = .5ex, bend right = 20,"\hat G"{name = U2,swap}]
			&
			\C E \ar[d,shift left = .5ex, bend left = 20,"\tilde F"{name = V1}]
			\ar[d,shift right = .5ex, bend right = 20,"\tilde G"{name = V2,swap}]
			\ar[r,"p"']
			&		
			\C B \ar[d,shift left = .5ex, bend left = 20," F"{name = W1}]
			\ar[d,shift right = .5ex, bend right = 20," G"{name = W2,swap}]
			\\[.5em] \C G' \ar[r,"Q'"]& \C E' \ar[r,"p'"]	& \C B'
			\ar[from = U1, to = U2,Rightarrow,shorten <= 2pt, shorten >= 2pt,"\hat\eta"]	
			\ar[from = V1, to = V2,Rightarrow,shorten <= 2pt, shorten >= 2pt,"\tilde\eta"]	
			\ar[from = W1, to = W2,Rightarrow,shorten <= 2pt, shorten >= 2pt,"\eta"]	
		\end{tikzcd} \]
	commutes.	
\item Both \( (\hat\eta,\eta) \colon (\hat F,F) \to (\hat G,G) \) and \( (\tilde\eta,\eta) \colon (\tilde F,F) \to (\tilde G,G) \) fulfil 
	\begin{align*} 
		&\overline{\eta_b}^{q}(\Phi) = \hat\eta_\Phi \quad\text{for }b \in \C B, \Phi \in \C G, \text{ such that }q(e) = b,\\
		&\overline{\eta_b}^{p}(e) = \tilde\eta_e \quad\text{for }b \in \C B, e \in \C E,\text{ such that }p(e) = b.
	\end{align*}
	\end{enumerate}
\end{defn}

\begin{defn} 
	Let \( I \colon \C C\opp \to \Cat, I' \colon \C D\opp \to \Cat \) be dep-categories and \( (F,F_{\spacer},F^{\spacer})\), \( (G,G_{\spacer},G^{\spacer}) \) be dep-functors.
	A \emph{dep-natural transformation} \( (F,F_{\spacer},F^{\spacer}) \To  (G,G_{\spacer},G^{\spacer})  \) is a transformation \( \eta \colon (F,F_{\spacer}) \To (G,G_{\spacer}) \) of the underlying indexed functors such that 
	\begin{enumerate}[leftmargin = 2.5em, itemindent = 1em]
	\renewcommand{\labelenumi}{(dep-\theenumi)}
	\setcounter{enumi}{\value{depcounter}}
	\item for all \(c \in \C C, \lambda \in I(c) \) and \( \Phi \in \dHom(c,\lambda) \) the equality \( G^{\lambda}(\Phi) \circ \eta_c = F^{\lambda}(\Phi) \) holds.
	\end{enumerate}
\end{defn}

\begin{lem}\label{2depnattotrafo} 
	Let \( I,I'\) be dep-categories, \( F,G \colon I \to I'\) be dep-functors and \( \eta \colon F \To G \) be a dep-natural transformation. 
	Then we obtain a transformation of the associated sas-maps, \( \sasmap \eta \colon \sasmap F \to \sasmap G	 \), where \( \sasmap F, \sasmap G \) are obtained as in Lemma \ref{2depfuntomap}.
\end{lem}

\begin{prf} 
	We first note that we obtain the transformation \((\tilde \eta,\eta) \colon ( \tilde F, F) \To (\tilde G,G) \) as in the standard result for indexed categories, thus in 
	\[ \begin{tikzcd}
			\C G \ar[r,"Q"'] \ar[d,shift left = .5ex, bend left = 20,"\hat F"{name = U1}]
			\ar[d,shift right = .5ex, bend right = 20,"\hat G"{name = U2,swap}]
			&
			\C E \ar[d,shift left = .5ex, bend left = 20,"\tilde F"{name = V1}]
			\ar[d,shift right = .5ex, bend right = 20,"\tilde G"{name = V2,swap}]
			\ar[r,"p"']
			&		
			\C B \ar[d,shift left = .5ex, bend left = 20," F"{name = W1}]
			\ar[d,shift right = .5ex, bend right = 20," G"{name = W2,swap}]
			\\[.5em] \C G' \ar[r,"Q'"]& \C E' \ar[r,"p'"]	& \C B'
			\ar[from = U1, to = U2,Rightarrow,shorten <= 2pt, shorten >= 2pt,"\hat\eta"]	
			\ar[from = V1, to = V2,Rightarrow,shorten <= 2pt, shorten >= 2pt,"\tilde\eta"]	
			\ar[from = W1, to = W2,Rightarrow,shorten <= 2pt, shorten >= 2pt,"\eta"]	
		\end{tikzcd} \]
	we already know that the right square commutes. 
	To see that the left square commutes we first have to define \( \hat \eta \).
	We set \( \hat \eta_\Phi = (\eta_c,1_{F_c(\lambda)}) \colon \hat F(\Phi) \to \hat G(\Phi) \)---which is possible as \( F(\Phi) = G(\Phi) \circ f \).
	This allows us to compute
	\[
		(Q' \ast \hat\eta)_\Phi =  Q'(\eta_c,1_{F_c(\lambda)}) = (\eta_c,1_{F_c(\lambda)}) = \tilde\eta_\lambda = \tilde\eta_{Q(\Phi)} = (\tilde\eta \ast Q)_\Phi. 	
	\]
	Additionally, we calculate that
	\[
		\overline{\eta_c}^{q}(\Phi) = \big((\eta_c,1_{F_c(\lambda)}) \colon G(\Phi) \circ \eta_c \to G(\Phi)\big) = \hat\eta_\Phi.\qedhere
	\]
\end{prf}
 
Obviously, we have the reverse direction as well.

\begin{lem}\label{trafoto2depnat}
	Let \( (Q,p),(Q',p') \) be sas-towers, \( \sasmap F, \sasmap G \colon (Q,p) \to (Q',p') \) be maps of sas-towers. Then any transformation \( \sasmap \eta \colon \sasmap F \to \sasmap G \) gives rise to a dep-natural transformation \( \eta \colon( F,F_{\spacer}, F^{\spacer}) \To (G,G_{\spacer}, G^{\spacer})\) between the associated dep-functors from Lemma \ref{sasmapto2depfun}.
\end{lem}

\begin{prf} 
We already know, that the underlying map of split fibrations \( (\tilde \eta, \eta) \) gives rise to a transformation of indexed functors.
Hence, it remains to confirm the remaining property, that 
\[
	G\lambda(\Phi) \circ \eta_c = F^\lambda(\Phi).
\]
But \( G_\lambda(\Phi) \circ \eta_c := (\eta_c)^{(p'\circ Q')*}\big(G_\lambda(\Phi)\big) \) and thus, we calculate
\begin{align*} 
	(\eta_c)^{(p'\circ Q')*}\big(G_\lambda(\Phi)\big) &= \hat F(\Phi) &&\text{as }\hat\eta_\Phi \colon \hat F(\Phi) \to \hat G(\Phi)
	\\
							   &= F_\lambda(\Phi) &&\text{by definition}.\qedhere
\end{align*}
\end{prf}

The functoriality of the assignment ``transformation of sas-maps'' \( \mapsto  \) ``dep-natural transformation'' is again immediate from the definition (similarly for the reverse direction).
Thus these assignment routines determine two 2-functors between categories we now define.

\begin{defn}[Category of sas-towers]
	We define the 2-category \( \sas \) of sas-towers as follows:
	\begin{itemize}
		\item Its objects (0-cells) are sas-towers as in Definition \ref{sastowerdef},
		\item its 1-maps are maps of sas-towers as in Definition \ref{sasmapsdef}.
		\item its 2-maps are transformations of sas-maps as in Definition \ref{sastrafodef}.
	\end{itemize}
	The composition of 1-cells and the vertical and  horizontal composition of 2-maps is defined component-wise.
\end{defn}

\begin{lem} 
	We obtain two 2-functors
	\begin{align*} 
		(Q_{\spacer},p_{\spacer}) \colon \depC \to \sas,&& I &\mapsto (Q,p) &&\text{as in Lemma \ref{2depAmbiLem::1}}
		\\
								       &&F  &\mapsto \sasmap F &&\text{as in Lemma \ref{2depfuntomap}}
		\\
								       &&\eta &\mapsto \sasmap \eta &&\text{as in Lemma \ref{2depnattotrafo}}
		\\
		\C C_{\spacer} \colon \sas \to \depC,&& (Q,p) &\mapsto I &&\text{as in Lemma \ref{sasto2dep}}
		\\
							 &&	\sasmap F &\mapsto F &&\text{as in Lemma \ref{sasmapto2depfun}}
		\\
							 &&\sasmap \eta &\mapsto \eta &&\text{as in Lemma \ref{trafoto2depnat}}		
	\end{align*}
	
\end{lem}

\begin{prf} 
	As we remarked, the functoriality with respect to 1-maps and 2-maps is immediate in both cases.
\end{prf}

\begin{thm} 
	The two 2-functors of the previous lemma establish a 2-equivalence
	\[
		\sas \simeq \depC.
	\]
\end{thm}

\begin{prf} 
	It is immediate that mapping a dep-category to the associated sas-tower, and that tower to its associated dep-category yields the same dep-category (up to renaming).
	An analogous result holds for the dep-functors and dep-natural transformations.

	Conversely, if we are given a sas-tower applying \( \C C_{\spacer} \) first and then \( (Q_{\spacer},p_{\spacer}) \) yields the same sas-tower (up to renaming), similar for sas-maps and transformations of such maps.
\end{prf}

\chapter{Dependent arrows and Sigma-objects}
\label{sec::depsig}

\begin{defn}
	Let \(I \colon \C C\opp \to \Cat \) be an indexed category with Sigma-objects.
	A \emph{(dep,$\Sigma$)-structure} on \( I \) consists of a dep-structure on the underlying indexed category in the sense of Definition~\ref{def: depstruc1} together with
	\begin{itemize}
	\item a \emph{second-projection-arrow} \( \pr_2^\lambda \in \dHom\Big(\Grothendieck{c}{\lambda},I(\pr_1^\lambda)(\lambda)\Big) \) for each \( c \in \C C, \lambda \in I(c) \), such that 
		\begin{enumerate}[leftmargin = 2.5em, itemindent = 1em]
		\renewcommand{\labelenumi}{(d$\Sigma$-\theenumi)}
		\item\label{2depSigmadefn::cd1} for every \( f \colon c' \to c \) the equality
	\(
		(\pr_2^\lambda)\circ \Big(\Grothendieck{\lambda}{f}\Big) = \pr_2^{I(f)(\lambda)}	
	\)
	holds.
\item \label{2depSigmadefn::cd2}	
	For every \( \mu \colon \lambda \to \eta \) the equality
		\[
			I(\pr_1^\lambda)(\eta) \circ \pr_2^{\lambda} = \pr_2^\eta\circ \Big(\Grothendieck{\lambda,\eta}{\mu}\Big)
		\]
		holds.
		\end{enumerate}
	\end{itemize}
An indexed category with Sigma-objects and a dep-structure is called a \emph{(dep, \( \Sigma \))-category}.
\end{defn}
\begin{exm}\label{exm::ring3} 
	The dep-structure on the indexed category of a ring from Examples~\ref{exm::ring1} and \ref{exm::ring2} can be extended to a (dep,$\Sigma$)-structure by setting \( \pr_2^{(a,b)} = \langle a-b\rangle \), the ideal generated by \( a-b \). For a proof that this constitutes a (dep,$\Sigma$)-structure see~\cite{ehrhardt2depCategories2024}. 
\end{exm}

The following theorem is stated in \cite[Thm.~3.5.8]{ehrhardt2depCategories2024}, and is a straightforward generalisation of \cite[Thm.~5.4]{petrakisCategoriesDependentArrows2023}.
It will be of importance in Section~\ref{sec: Comparison}.

\begin{thm} 
	Every indexed category \( I \colon \C C\opp \to \Cat \) with Sigma-objects has an induced dep-structure of the following form:
	\begin{itemize}
		\item The dependent arrows in \( \dHom(c,\lambda) \) are arrows \( \Phi \colon c \to \Grothendieck{c}{\lambda} \) in \( \C C \) such that \( \pr_1^\lambda \circ \Phi = 1_c \).
		\item The precomposition of such a \( \Phi \) with an arrow \( f \colon c' \to c \) is defined via a pullback, as follows:
			\[ \begin{tikzcd}[execute at end picture={
			\rueckzug{2-2}{2-3}{3-2}{};
		}]
				c' \ar[ddr,"1_{c'}"{swap, near start},to path = {(\tikztostart.south) |- (\tikztotarget.west) \tikztonodes}, rounded corners]
				\ar[drr,"\Phi \circ f"{near start},to path = {(\tikztostart.east) -| (\tikztotarget.north) \tikztonodes}, rounded corners]
				\ar[dr,"\Phi \circ f"{description}, dotted]
				\\
				& \Grothendieck{c'}{I(f)(\lambda)} \ar[d,"\pr_1^{I(f)(\lambda)}"']  
					\ar[r,"\Grothendieck{\lambda}{f}"] 
				& \Grothendieck{c}{\lambda} \ar[d,"\pr_1^\lambda"]
				\\
				&c' \ar[r,"f"'] & c.
			\end{tikzcd} \]
			Note that \( \Phi \circ f \) on the top-most arrow is only the ordinary composition of arrows in \( \C C \), which is distinct from the precomposition \( \Phi \circ f \)!
		\item The postcomposition of \( \Phi \) with \( \mu \colon \lambda \to \eta \) is defined via 	
			\[
				\mu \circ \Phi := \Grothendieck{\lambda,\eta}{\mu} \circ \Phi.
			\]
			Again, the composition on the right is the ordinary composition in \( \C C \).
		\item The second-projection-arrow of \( \lambda  \) is defined via the pullback
			\[ \begin{tikzcd}[execute at end picture={
			\rueckzug{2-2}{2-3}{3-2}{};
		}]
				\Grothendieck{c}{\lambda}
				\ar[ddr,"1_{\Grothendieck{c}{\lambda}}"{swap, near start},to path = {(\tikztostart.south) |- (\tikztotarget.west) \tikztonodes}, rounded corners]
				\ar[drr,"1_{\Grothendieck{c}{\lambda}}"{near start},to path = {(\tikztostart.east) -| (\tikztotarget.north) \tikztonodes}, rounded corners]
				\ar[dr,"\pr_2^\lambda"{description}, dotted]
				\\
				& \Grothendieck{\Grothendieck{c}{\lambda}}{I(\pr_1^\lambda)(\lambda)} \ar[d,"\pr_1^{I(\pr_1^\lambda)(\lambda)}"'] 
				\ar[r,"\Grothendieck{\lambda}{\pr_1^\lambda}"]
				&\Grothendieck{c}{\lambda} \ar[d,"\pr_1^\lambda"]
				\\
				& \Grothendieck{c}{\lambda} \ar[r,"\pr_1^\lambda"'] & c.
			\end{tikzcd} \]
	\end{itemize}
\end{thm}

As we saw in Theorem~\ref{lem: equivSigmaComprehension} with Sigma-objects are biequivalent to split comprehension categories, that is adding Sigma-objects corresponds to adding a comprehension functor to the split fibration.
\[
	I\colon \C C\opp \to \Cat \textcolor{red}{+ \text{ Sigma-objects }} \textcolor{blue}{+ \text{ dep-structure}} \leftrightarrow 
	\begin{tikzcd} 
		|[color = blue]|	\C D \ar[r,"Q",blue] \ar[dr,"q"',bend right = 45,blue] & 	\C E\ar[d,"p"'] \ar[r,"P", red] & |[color = red]| \C C^\to \ar[dl,"\codom", red,bend left = 40]
\\
					 & \C C\end{tikzcd}
\]
It remains to examine to which additional structure on comprehension the dep-structures correspond. 
We will show that they correspond to \emph{higher comprehension categories}, which we will introduce in the remainder of this section.
Thus, the resulting picture is the following:
\[
	I\colon \C C\opp \to \Cat + \text{ Sigma-objects }\textcolor{red}{+ \text{(dep,$\Sigma$)-structure }} \leftrightarrow 
	\begin{tikzcd} 
		\C D \ar[r,"Q"', bend right = 30] \ar[r,Rightarrow,"\chi",shorten = 8pt,red]
			\ar[dr,"q"', bend right = 45]
			& \C E \ar[l,"\pr_2"', bend right = 30,red] 
			\ar[r,"P"] \ar[d,"p"]
			& \C C^{\to}
			\ar[dl,"\codom",bend left = 40]
			\\
			& \C C\end{tikzcd}
\]

Before we define higher comprehension categories we will define a property of natural transformations related to fibrations, and prove some results regarding this property, the importance of which will become apparent after the definition of higher comprehension categories.

\begin{defn}
	Let \( \C C, \C E, \C B \) be categories, \( p \colon \C E \to \C B \) a split fibration and \( F,G \colon \C C \to \C B, F',G' \colon \C C \to \C E  \) be functors connected by natural transformations \( \chi,\eta \), as in 
	\begin{equation} \begin{tikzcd}
		&[2em] \C E \ar[dd,"p"] 
		\\[-1em]
		\C C \ar[ur,"F'"{name = U}, bend left = 40] \ar[ur,"G'"'{name = V}, bend right = 0] 
		\ar[dr,"F"{name = W}, bend left = 0] \ar[dr,"G"'{name = S}, bend right = 40]
		\\[-1em]
		& \C B.
		\ar[from = U, to = V, "\chi", Rightarrow, shorten = 3pt]
		\ar[from = W, to = S, "\eta", Rightarrow, shorten = 3pt]
	\end{tikzcd}\label{NattrafoFibLift} \end{equation}
	We then say that \( \chi \) is the \emph{p-lift of} \( \eta \) if the above diagram commutes on the level of functors, and for all objects \( c \) of \( \C C \) we have that
	\[
		\overline{\eta_c}\big(G(c)\big) = \chi_c.
	\]
	In this case we write \( \overline{\eta} = \chi \) or \( \overline{\eta}^p = \chi\).
\end{defn}

First we show that the notation \( \overline{\eta} = \chi \) is actually justified, that is, if both \( \overline{\eta} = \chi \) and \( \overline{\eta} = \xi \), then \( \chi = \xi \).

\begin{lem} 
Given categories, \( \C C, \C E, \C B \), functors \( F,G \colon \C C \to \C B, F',G' \colon \C C \to \C E \), natural transformation \( \chi \colon F' \To G', \eta \colon F \To G \) and a fibration \( p \colon \C E \to \C B \) as above be given.
If \( \overline{\eta} = \chi \) and \( \overline{\eta} = \xi \), then 
 \( \chi = \xi \).
\end{lem}

\begin{prf} 
	This is immediate, as for all \( c \) in \( \C C \) we can compute
	\[\xi_e = \overline{\eta_c}G(c) = \chi_c. \qedhere\]
\end{prf}

\paragraph{Some notation for specific lifts}
When working with fibrations \( p \colon \C E \to \C B \) we already saw that the precomposition of a vertical arrow \( \eta \colon e' \to e \) in \( \C E \) with an arrow \( f \) of the base category involves invoking the universal property as such:
\[ \begin{tikzcd}
	f^*(e') \ar[d,dotted,"h"] \ar[r,"\overline{f}(e')"] & e' \ar[d,"\eta"] 
	\\
	f^*(e) \ar[r,"\overline{f}(e)"'] & e,	
\end{tikzcd} \]
where \( h \) is then \( \eta \bullet f \). 
From now on we will always use \( \eta \bullet f \) for the arrow obtained in this manner.

Another issue left regards arrows vertical with respect to a fibration.
As discussed earlier, the factorisation system obtained from a fibration distinguishes between arrow which are vertical, in the sense that they map to an isomorphism under \( p \), and those that are cartesian, that is they are isomorphic to a chosen \( p \)-cartesian lift.
However, when discussing fibrations stemming from indexed categories this distinction is too coarse for us, it merely distinguishes morphisms \( (f,\eta) \) where either \( f \) is an isomorphism or \( \eta \) is. 
What we want, however, is that either \( f \) is an identity or \( \eta \) is.
\begin{itemize}
	\item\index{p-prone part@$p$-prone part! of a morphism|hyperit} 
For \( \eta \), this can obtained easily, just consider \( \overline{p(f,\eta))}^p(e) \), which will have \( \eta \) ``removed''. 
Similarly, this works for any morphism $f\colon e' \to e$ in the total category of a (split) fibration \( p \colon \C E \to \C B \), taking \( \overline{p(f)}(e) \).
We will call this the \emph{$p$-prone part} of \( f \)---a slight abuse of the terminology coined by Johnstone and Taylor in \cite[265ff]{johnstoneSketchesElephantTopos2002} and \cite[Definition~9.2.6]{taylorPracticalFoundationsMathematics1999} respectively---and denote it as \( \prone f \).
\item\index{true vertical!of a morphism|hyperit}
For \( f \) we need to work more:
considering the diagram
\[ \begin{tikzcd}
	\lambda \ar[dr,"{(f,\eta)}"]  \ar[d,"h"', dotted]
	\\
f^*(\mu) \ar[r,"\overline{f}(\mu)"'] & \mu,
\end{tikzcd} \]
we can see that \( h \) must be \( \eta \).
Inspired from this, for an arbitrary fibration \( p \colon \C E \to \C B \) and a vertical morphism \( f \colon e' \to e \), we define through the same universal property
\[ \begin{tikzcd}
	e' \ar[dr,"f"] \ar[d,"\tv{f}"'] 
	\\
	\big(p(f)\big)^*(e) \ar[r,"\overline{p(f)}(e)"'] & e,
\end{tikzcd} \]
the \emph{true vertical} \( \tv f \) of \( f \).
It is immediate from this definition that \( p(\tv f) = 1_{p(e')} \), as
\[
	p(f) = p\big(\overline{p(f)}(e) \circ \tv f\big) = p\big(\overline{p(f)}(e)\big) \circ p\big(\tv f\big) = p(f) \circ p\big( \tv f\big),
\]
hence, as \(p( f) \) is invertible, the claim follows.
\end{itemize}
Note however, that while any morphism in the total category of a fibration can be decomposed into \( \prone f \circ \tv f \), this does not always yield the desired outcome, as for example \( \tv f \) may not lie over the identity, since the proof of this fact rested on \( p(f) \) being invertible!

The following lemma will be needed in order for the definition of higher comprehension categories to be well-defined.

\begin{lem} 
	If \( P \colon \C E \to \C B^{\to} \) is a comprehension category, then \( f \bullet P_0(f) \) is vertical if \( f \) is.
\end{lem}

\begin{prf} 
	First, we use that any pullback of an isomorphism is an isomorphism, so \( P_0(f) \) is an isomorphism.
	Then \( f \bullet P_0(f) \) is vertical because \( p\big(f \bullet P_0(f)\big) = P_0(f) \).
\end{prf}

\newcounter{hCCdefn}
\begin{defn}[Higher comprehension categories]
	\index{Higher comprehension category|hyperit}
	We define a \emph{higher comprehension category} to consist of 
	\begin{enumerate}[itemindent = 2.5em]
		\renewcommand{\labelenumi}{(hCC-{\theenumi})}
		\item a sas-tower (Definition~\ref{sastowerdef}) \( (Q \colon \C G \to \C E, p \colon \C E \to \C B) \),
		\item a comprehension category \( P \colon \C E \to \C B^{\to} \) such that \( \codom \circ P = p \) and 
		\item\label{hCCdefn::1} a functor \( \pr_2 \colon \C E \to \C G \) together with a natural transformation \( \chi \colon Q \circ \pr_2 \To \id_{\C E} \) such that
			\begin{enumerate}
			\renewcommand{\labelenumii}{\theenumii)}
			\item for all \( e \in \C E \) and all \( f \colon c' \to p(e) \) we have 
				\[
					\pr_2\big(\overline{f}^p(e)\big)= \overline{P_0\big(\overline{f}^p(e)\big)}^{q}(\pr_2(e)).
				\]
			\item for all \( e,e' \in \C E \) and all \( f \colon e' \to e \) vertical with respect to \( p \) we have
				\[
				\underline{\tv f \bullet P(e')}_Q\pr_2(e') = \overline{P_0(f)}^{q}\big(\pr_2(e)\big).
				\]
			\item \( \overline{P}^p = \chi \).
	\end{enumerate}
			\setcounter{hCCdefn}{\value{enumi}}
	\end{enumerate}
 We will denote this data as \( (Q,P,\chi) \).
\end{defn}

\begin{lem} \label{2depSigmatohCC}
	Every (dep,$\Sigma$)-category \( I \colon \C C\opp \to \Cat \) gives rise to a higher comprehension category \( (Q,P,\chi) \) in the following way:
	the underlying comprehension category \( P \colon \C E \to \C C^\to \) is defined from the underlying indexed category with Sigma-objects.
	\begin{itemize}
		\item 
	The category \( \C D \) has objects all dep-arrows \( \Phi \), and an arrow \( \Phi \to \Psi \) for \( \Phi \in \dHom(c,\lambda) \), \( \Psi \in \dHom(c',\eta) \) is a pair \( (f,\mu) \) of an arrow \( f \colon c \to c'  \) in  \( \C C \) and an arrow \( \mu \colon \lambda \to F(f)(\eta) \) such that \( \Psi \circ f = \mu \circ \Phi \).
\item The functor \( Q \) maps \( \Phi \in \dHom(c,\lambda) \) to \( \lambda \) and \( (f,\eta) \colon \Phi \to \Psi \) to \( (f,\eta) \).
\item The functor \( \pr_2 \) maps \( \lambda \) to \( \pr_2^\lambda \) and \( (f,\mu) \colon \lambda \to \eta \) to
	\[
		\Big( \big( \Grothendieck{\eta}{f}\big) \circ \Grothendieck{\lambda,\eta}{\mu}, F(\pr_1^\lambda)(\mu)\Big) \colon \pr_2^\lambda \to \pr_2^\eta.
		\]
	\item The natural transformation \( \chi \colon Q\circ \pr_2 \to \id_\C E\) is defined via \( \chi_\lambda := (\pr_1^\lambda , 1_{F(\pr_1^\lambda)(\lambda)}) \).
\end{itemize}
\end{lem}

\begin{lem}\label{HigherCCto2depS} 
	We can map any higher comprehension category \( (Q,P,\chi) \) to a (dep,$\Sigma$)-category by letting the underlying dep-category be the category \( \C C \) obtained from the sas-tower \( (Q,p) \) by Lemma~\ref{sasto2dep} and the (dep,$\Sigma$)-structure, that is the Sigma-objects are obtained from the comprehension functor \( P \).
	For each \( e \in \C E \) we define \( \pr_2^{e} := \pr_2(e) \).
\end{lem}

\begin{prf} 
	As the two lemmata already yield the underlying dep-structure, it remains to check conditions \itemref{d$\Sigma$}{2depSigmadefn::cd1} and \itemref{2d$\Sigma$}{2depSigmadefn::cd2}.
	For the first condition, we calculate
	\begin{align*}
		[\pr_2^e]\big(\Grothendieck{e}{f}\big) &= \overline{P_0(\overline{f}^p(e))}^{*(q)}(\pr_2(e))
						    = \pr_2\big(\overline{f}^p(e)\big)
						    \\ &= \pr_2^{e \circ f}.
	\end{align*}
	For the second condition, we calculate (where \(f \colon e' \to e \) is such that \( p(f) = 1_{p(e)} \)
	\begin{align*} 
		(f \bullet \pr_1^{e'}) \circ \pr_2^{e'} &= \big(\tv f \bullet P_0(e')\big) \circ \pr_2(e')
						     = \underline{\tv f \bullet P_0(e')}_Q \big( \pr_2(e')\big)
						     \\ &= \overline{P_0(f)}^{q}\big(\pr_2(e)\big)
						     = \overline{\Grothendieck{p(e)}{f}}^{q} \big( \pr_2(e)\big)
						     \\ &= [\pr_2^e]\Big(\Grothendieck{p(e)}{f}\Big).\qedhere
	\end{align*}
\end{prf}

\begin{defn}[Maps between higher comprehension categories]
	Let \( (Q,P,\chi) \), \( (Q',P',\chi') \) be higher comprehension categories. 
	A \emph{map} is a triple \( (\hat F, \tilde F, F)\colon (Q,P,\chi) \to (Q',P',\chi') \) such that 
	\begin{enumerate}[itemindent = 2.7em]
		\renewcommand{\labelenumi}{(hCC-{\theenumi})}
		\setcounter{enumi}{\value{hCCdefn}}
	\item \( (\hat F, \tilde F, F) \) is a map of sas-towers (see Definition~\ref{sasmapsdef}).
	\item \(( \tilde F, F) \) is a strict map of comprehension categories (see~\cite[Def.~2.2]{ahrensComparingSemanticFrameworks2025a}).
	\item \( \pr_2' \circ \tilde F = \hat F \circ \pr_2 \) and \( \chi' \ast \tilde F = \tilde F \ast \chi \).
	\end{enumerate}
\end{defn}

\begin{lem}\label{2deptoHCCfun}
	Any (dep,$\Sigma$)-functor \( F \colon I \to I \) can be translated into a map between the associated higher comprehension categories (from Lemma~\ref{2depSigmatohCC}).
\end{lem}

\begin{prf} 
	From Lemma~\ref{2depfuntomap} we know that the underlying 2-dep-functor \( F \) yields a map \( \sasmap F\colon (Q,p) \to (Q',p') \) of sas-towers.
	Analogously, we know that the underlying (2-fam,$\Sigma$)-functor yields a strict map \( (\dot F, F) \) of comprehension categories.
	It is immediate from the definition of those maps that \( \dot F = \tilde F \), so the maps agree on the indexed category structure.
	To see that \( \pr_2' \circ \tilde F = \hat F \circ \pr_2 \) we simply calculate for \( (f,\eta) \colon \lambda \to \mu \), where \( \lambda \in \fHom(c), \mu \in \fHom(c') \)
	\begin{align*}
		(\pr_2' \circ \tilde F)(\lambda) &= \pr_2^{F_c(\lambda)} = F_{\lambda \circ \pr_1^\lambda}(\pr_2^\lambda) = (\hat F \circ \pr_2)(\lambda) \\
		\hbox{and } (\pr_2' \circ \tilde F)(f,\eta)  &= \pr_2'\big((F(f),F_c(\eta)\big)
							   \\&= \Big(\Grothendieck{F_{c'}(\mu)}{F(f)} \circ \Grothendieck{F_c(\lambda),F_{c'}(\mu)}{F_c(\eta)}, F_c(\eta) \bullet F_{\Grothendieck{c}{\lambda}}(\pr_1^\lambda)\Big)
											  \\ &= \Big( F\Big(\Grothendieck{\mu}{f} \circ \Grothendieck{\lambda,\mu}\Big),F_c(\eta\bullet \pr_1^{\lambda}\Big) 
											  \\ &= \hat F\big(\pr_2(f,\eta)\big)
	\end{align*}
This shows the last remaining property, so \( (\hat F, \tilde F, F)  \) is a map between higher comprehension categories.	
\end{prf}

\begin{lem}\label{HCCto2depSFun}
	If \( (\sasmap F) \colon (Q,P,\chi) \to (Q',P',\chi') \) is a map between higher comprehension categories, then \( F  \) can be made a (dep,$\Sigma$)-functor of the associated (dep,$\Sigma$)-categories from Lemma~\ref{HigherCCto2depS} by setting 
	\[
		F_c(\lambda) := \tilde F(\lambda) \quadd\hbox{and}\quadd F_\lambda(\Phi) := \hat F(\Phi).
	\]
\end{lem}

\begin{prf} 
	It is immediate that \( F \) obtained in this manner is both a map of the underlying indexed categories with Sigma-objects and a dep-functor of the underlying dep-categories from Lemma~\ref{sasmapto2depfun}.
	Thus by showing 
	\begin{align*} 
		F(\pr_2^\lambda) = \hat F\big(\pr_2(\lambda)\big) = \pr_2'\big(\tilde F(\lambda)\big) = \pr_2^{F(\lambda)}
	\end{align*}
	we have that \( F \) is a (dep,$\Sigma$)-functor.
\end{prf}

\begin{defn}[Transformation of maps between higher comprehension categories]
	Let \( \sasmap F, \sasmap G \colon (Q,P,\chi) \to (Q',P',\chi') \) are two maps between higher comprehension categories, we define a \emph{transformation} to be a transformation between maps of sas-towers (see Definition~\ref{sastrafodef}) that is also a transformation of maps between the underlying comprehension categories (see~\cite[Def.~2.2]{ahrensComparingSemanticFrameworks2025a})
\end{defn}

\begin{lem} 
	Any (dep,$\Sigma$)-natural transformation between (dep,$\Sigma$)-functors \( F,G \colon \C C \to \C D \) induces a transformation between the associated maps of higher comprehension categories from Lemma~\ref{2deptoHCCfun}. 
\end{lem}

\begin{prf}
	Immediate from Lemma~\ref{2depnattotrafo} and the translation of maps between indexed functors and maps between comprehension categories.
\end{prf}

The reverse direction holds trivially as well.

\begin{lem} 
	Any transformation of maps \( \sasmap F \To \sasmap G \) induces a (dep,$\Sigma$)-natural transformation between the associated (dep,$\Sigma$)-functors of Lemma~\ref{HCCto2depSFun}.
\end{lem}

\begin{prf} 
	Immediate from Lemma~\ref{trafoto2depnat} and the translation of maps between indexed functors and maps of comprehension categories.
\end{prf}

\chapter{A comparison with generalised categories with families}\label{sec: Comparison}
\label{ComparisongCwF2depS}
Coraglia and Di Liberti introduced a notion very similar looking to higher discrete comprehension categories---generalised categories with families---in \cite{coragliaContextJudgementDeduction2024}. 
Coraglia and Emmenegger then showed in \cite{coraglia2categoricalAnalysisContext2024} that generalised categories with families are biequivalent to comprehension categories. 

We start by recalling the definition.

\begin{defn}[generalised categories with families]
	A \emph{generalised category with families} consists of a map of  fibrations \( \Sigma \colon \big(\dot u \colon \dot{\mathcal{U}} \to \C B\big) \to \big( u \colon \mathcal{U} \to \C B\big) \) together with an adjunction
	\[ \begin{tikzcd}
		\dot{\mathcal{U}} \ar[r,"\Sigma"'{name = U}, bend right = 30] 
		\ar[dr,"\dot u"', bend right = 40]
		& \mathcal{U} \ar[l,"\Delta"'{name = V}, bend right = 30]
		\ar[d," u"]
		\\
		& \C B
		\ar[from = U, to = V, "\dashv"{description, sloped},phantom]
	\end{tikzcd} \]
	such that all components of the unit and counit are cartesian with respect to \( \dot u \) and \( u \) respectively.
\end{defn}

Looking at the diagrams governing the two notions they appear eerily similar:
\begin{table}[h]
	\caption{gcwfs vs. hccs}
	\label{table::gcwf}
	\centering
	\begin{tabular}{p{0.4\textwidth}p{0.4\textwidth}}
		\toprule
		generalised categories with families & higher comprehension categories
		\\\cmidrule(lr){1-2}
	\[ \begin{tikzcd}
		\dot{\mathcal{U}} \ar[r,"\Sigma"{swap,name = U}, bend right = 30] \ar[dr,"\dot{u}"', bend right = 45]
		& \mathcal{U} \ar[l,"\Delta"{swap,name = V}, bend right = 30]
		\ar[d,"u"]
		\\
		& \C B
		\ar[from = U, to = V, "\dashv"{sloped, description},phantom]
	\end{tikzcd} \]
	
						     &
		\[ \begin{tikzcd}
			\C G \ar[r,"Q"', bend right = 30] \ar[r,Rightarrow,"\chi",shorten = 8pt]
			\ar[dr,"q"', bend right = 45]
			& \C E \ar[l,"\pr_2"', bend right = 30] 
			\ar[r,"P"] \ar[d,"p"]
			& \C B^{\to}
			\ar[dl,"\codom",bend left = 40]
			\\
			& \C B
		\end{tikzcd} \]
		\\\bottomrule
	\end{tabular}
\end{table}

However, there are some differences, which we will highlight in the following paragraphs, aiding our analysis how  generalised categories relate to categories with dependent arrows, and how the biequivalence between comprehension categories and generalised categories with families fits into our picture.
\par\bigskip
\paragraph{Adjunction versus counit}
	The immediately visible difference (discarding the triangle on the right for higher comprehension categories) is that where generalised categories with families demand an adjunction, higher discrete comprehension categories demand only a natural transformation.
	Unpacking the definition of an adjunction one sees that an adjunction is a stricter version, as it decomposes to two natural transformations, 
	\[
		\hbox{the unit } \id_{\mathcal{U}} \To \Delta\circ \Sigma \hbox{ and the counit }  \Sigma\circ \Delta \To \id_{\dot{\mathcal{U}}}, 
	\]
hence, our natural transformation \( \chi \) plays the role of a counit.
\par\bigskip
\paragraph{Split versus unsplit}
One difference between generalised categories with families is that they are concerned with normal fibration, whereas we demand that all fibrations come equipped with a normalised cleavage, i.e. are split.
This is the reason we demand that all splittings are preserved by the functors involved, more precisely:
\begin{itemize}
	\item in the definition of sas-towers, the equivalent part to \itemref{sas}{sastowerdef::item3} in the definition of generalised categories with families is that \( \Sigma \) is a functor between fibrations and hence preserves cartesian morphisms. 
		The condition equivalent to \itemref{sas}{sastowerdef::item4} is, that it reflects cartesian morphisms, a fact that can be deduced from the definition (see \cite[Lemma 3.20.3]{coraglia2categoricalAnalysisContext2024}).
		The part of oplifts with respect to vertical arrows (with respect to \( u \)) is obviously missing, but we will explain later why this does not matter for the translation to comprehension categories.
	\item In condition \itemref{hCC}{hCCdefn::1} in the definition of higher comprehension categories the part c) corresponds to the counit consisting of (chosen) lifts of arrows in \( \C E \).
		That the counit for the adjunction in a generalised category with families is at least cartesian (as no specific chosen lifts are given) is guaranteed by  the definition.
	\item That \( \pr_2 \) preserves our choices of lifts for \( p \) corresponds to \( \Delta \) respecting the cartesianness of morphisms, this is guaranteed by \cite[Lemma 3.20.2]{coraglia2categoricalAnalysisContext2024}.
\end{itemize}
\paragraph{The (bi-/2-)equivalences}
The biequivalences described in \cite{coraglia2categoricalAnalysisContext2024} are between comprehension categories and weakening comonads, and between weakening comonads and generalised categories with families. 
Extending this picture with the notions of categories with families/dependent arrows and the corresponding equivalence proved earlier, we get 
\begin{equation} \begin{tikzcd}
	\depSC \ar[r,"\text{\footnotesize 2-$\simeq$}"{description},phantom] \ar[r,shift right = 1.3ex]
 &\HCompCat \ar[drr,dashed,no head,"?", bend left = 20] \ar[l,shift right = 1.3ex]
 \\
 \Sigma-\mathsf{IC} \ar[r,shift left = 1.3ex] \ar[u] \ar[r,"\text{\footnotesize 2-$\simeq$}"{description},phantom] 
	& \CompCat \ar[l,shift left = 1.3ex] \ar[r,"\text{\footnotesize bi-$\simeq$}"{description}, phantom] \ar[r,shift right = 1.3ex] \ar[u]
	& \WCmd \ar[l,shift right = 1.3ex] \ar[r,"\text{\footnotesize bi-$\simeq$}",description, phantom] \ar[r,shift right = 1.3ex] 
	& \gCwF, \ar[l,shift right = 1.3ex]
	\label{gCwFhCCequiv}
\end{tikzcd} \end{equation}
where the relation indicated with ``?'' is the one we want to explore in the following paragraphs.
For this reason we compute the obtained generalised category with families from a (2-fam,$\Sigma$)-category \( \C C \) by applying all the functors in the above diagram.
This generalised category with families has as \( \dot{\mathcal{U}} \) the category \( \C E \) of family arrows and 2-family arrows, 
and \( \mathcal{U} \) is the category of \( K \)-coalgebras, for the weakening comonad \( (K,\epsilon,\nu) \) defined via
\begin{enumerate}
	\item\label{weakCmd::cd1} \( K(\lambda) = \lambda \circ \pr_1^\lambda \) for \( \lambda \in \fHom(c) \), and for \( (f,\eta) \colon \lambda \to \mu \) 
		\begin{equation}
			K(f,\eta) =\Big (\Grothendieck{\mu}{f}\circ \Grothendieck{\lambda,\mu}{\eta},\eta \bullet \pr_1^\lambda\Big).
		\label{genCatFam::eq1}\end{equation}
		The well-definedness can be seen from the commutativity of 
		\[ \begin{tikzcd}
			\Grothendieck{c}{\lambda} \ar[r,"\Grothendieck{\lambda,\mu \circ f}{\eta}"]
			\ar[d,"\pr_1^\lambda"']
			& \Grothendieck{c}{\mu \circ f} \ar[r,"\Grothendieck{\mu}{f}"]
			\ar[d,"\pr_1^{\mu \circ f}"]
			& \Grothendieck{c'}{\mu} \ar[d,"\pr_1^\mu"]
			\\
			c \ar[r,"1_c"'] & c \ar[r,"f"'] & c',
		\end{tikzcd} \]
			which yields \(\mu \circ \pr_1^\mu \circ \Grothendieck{\mu}{f} \circ \Grothendieck{\lambda,\mu \circ f}{\eta} = \mu \circ \pr_1^\lambda.\)
			
	\item \( \epsilon_\lambda =  (\pr_1^\lambda,1_{\lambda \circ \pr_1^{\lambda}}) \colon \lambda \circ \pr_1^\lambda \to \lambda \)
	\item\label{weakCmd::cd3} \( \nu_\lambda = (Q, 1_{\lambda \circ \pr_1^{\lambda} \circ \pr_1^{\lambda \circ \pr_1^{\lambda}}}) \colon \lambda \circ \pr_1^{\lambda} \to \lambda \circ \pr_1^\lambda\circ \pr_1^{\lambda \circ \pr_1^{\lambda}}  \) where \( Q \) stems from the pullback diagram
				\begin{equation} \begin{tikzcd}
					\Grothendieck{c}{\lambda} \ar[ddr,to path={(\tikztostart.south) |- (\tikztotarget.west) \tikztonodes}, rounded corners,"\id_{\Grothendieck{c}{\lambda}}"']
					\ar[drr,to path = {(\tikztostart.east) -| (\tikztotarget.north) \tikztonodes}, rounded corners, "\id_{\Grothendieck{c}{\lambda}}"]
							\ar[dr,"Q" description, dotted]
							\\
							&\Grothendieck{\Grothendieck{c}{\lambda}}{\lambda \circ \pr_1^{\lambda}} \ar[d,"\pr_1^{\lambda \circ \pr_1^{\lambda}}"'] \ar[r,"\Grothendieck{\lambda}{\pr_1^\lambda}"]
							& \Grothendieck{c}{\lambda} \ar[d,"\pr_1^{\lambda}"]
							\\
							&\Grothendieck{c}{\lambda} \ar[r,"\pr_1^{\lambda}"'] & c.
				\end{tikzcd}\label{genCatFam::diag1} \end{equation}
				
\end{enumerate}
This category \( \CoAlg(K) \) can be spelled out explicitly. It has as
\begin{itemize}
	\item objects arrows \( (\Phi,\eta) \colon \lambda \to \lambda \circ \pr_1^\lambda\) such that 
		\[ \begin{tikzcd}
			\lambda \ar[dr,"{(1_c,1_\lambda)}"'] \ar[r,"{(\Phi,\eta)}"] & \lambda \circ \pr_1^\lambda  \ar[d,"{(\pr_1^\lambda,1_\lambda)}"]
						    \\ & \lambda
		\end{tikzcd}
	\quad\hbox{and}\quad
\begin{tikzcd} 
	\lambda \ar[r,"{(\Phi,\eta)}"] \ar[d,"{(\Phi,\eta)}"']
	& \lambda \circ \pr_1^\lambda \ar[d,"\nu_\lambda"]
	\\
	\lambda \circ \pr_1^\lambda \ar[r,"{K(\Phi,\eta)}"'] & \lambda \circ \pr_1^{\lambda} \circ \pr_1^{\lambda \circ \pr_1^{\lambda}}
\end{tikzcd}
\]
		commute.
		The commutativity of the first diagram simply spells out that both \( \pr_1^\lambda \circ \Phi = 1_c \) and \( (1_\lambda \bullet \pr_1^{\lambda}) \circ \eta = 1_\lambda \), which simplifies to \( \eta = 1_\lambda \).
		This in turn can be used for the commutativity of the second diagram, which entails both
		\[
			Q \circ \Phi = \Grothendieck{\lambda \circ \pr_1^\lambda}{\Phi} \circ \Grothendieck{\lambda,\lambda \circ \pr_1^{\lambda}}{\eta}\circ \Phi  \quad\hbox{and}\quad \eta \bullet(\pr_1^\lambda\circ \Phi) \circ \eta = 1_{\lambda \circ \pr_1^\lambda \circ \pr_1^{\lambda\circ \pr_1^{\lambda}}} \bullet \Phi \circ \eta.
		\]
		We immediately see that the left hand side of the second equality simplifies to \( \eta \circ \eta = 1_\lambda \), and the right hand side simplifies to \( 1_\lambda \) as well, so this equality always holds.
		For the first equality, we observe that as \( \eta = 1_\lambda \) we have \( \Grothendieck{\lambda,\lambda \circ \pr_1^\lambda}{\eta} = 1_{\Grothendieck{c}{\lambda}} \), and we can extend the pullback diagram \eqref{genCatFam::diag1} to 
		\[ \begin{tikzcd}[execute at end picture={
			\rueckzug{2-2}{2-3}{3-2}{};
			\rueckzug{2-3}{2-4}{3-3}{};
		}]
			c \ar[r,"\Phi"]\ar[ddr,to path={(\tikztostart.south) |- (\tikztotarget.west) \tikztonodes}, rounded corners,"1_c"']\ar[dr,dashed]
							&[3em] \Grothendieck{c}{\lambda} \ar[d,"1_{\Grothendieck{c}{\lambda}}"]		\ar[drr,to path = {(\tikztostart.east) -| (\tikztotarget.north) \tikztonodes}, rounded corners, "\Phi"]
							\ar[dr,"Q", dotted, to path ={(\tikztostart.south east) -- (\tikztotarget.north west)\tikztonodes}]
							&[3em]
						\\[1em]
							&\Grothendieck{c}{\lambda} \ar[r,"\Grothendieck{\lambda \circ \pr_1^\lambda}{\Phi}"'] \ar[d,"\pr_1^\lambda"']
						&\Grothendieck{\Grothendieck{c}{\lambda}}{\lambda \circ \pr_1^{\lambda}} \ar[d,"\pr_1^{\lambda \circ \pr_1^{\lambda}}"'] \ar[r,"\Grothendieck{\lambda}{\pr_1^\lambda}"]							& \Grothendieck{c}{\lambda} \ar[d,"\pr_1^{\lambda}"]
							\\
							&c \ar[r,"\Phi"']	&\Grothendieck{c}{\lambda} \ar[r,"\pr_1^{\lambda}"'] & c.
				\end{tikzcd} \]
It is immediate that the dashed arrow has to be \( \Phi \), which yields the required equality.
Thus \( (\Phi,\eta) \) can be reduced to \( \Phi \), such that 
\[ \begin{tikzcd}
	c \ar[r,"\Phi"] \ar[dr,"1_c"'] & \Grothendieck{c}{\lambda} \ar[d,"\pr_1^{\lambda}"]
	\\
				       & c
\end{tikzcd} \]
commutes.
\item An arrow \( (f,\eta) \colon (\Phi \colon c \to \Grothendieck{c}{\lambda}) \to (\Psi\colon d \to \Grothendieck{d}{\mu})\) is an arrow \( f \colon c \to d  \) together with a 2-fam-arrow \( \eta \colon \lambda \To \mu \circ f \) such that 
	\[ \begin{tikzcd}
	\lambda \circ \pr_1^\lambda \ar[from = r,"{(\Phi,1_\lambda)}"] \ar[d,"{(\Grothendieck{\mu f} \circ \Grothendieck{\lambda,\mu}{\eta}, \eta \bullet \pr_1^\lambda)}"']
		& \lambda \ar[d,"{(f,\eta)}"]
		\\
	\mu \circ \pr_1^{\mu} \ar[from = r,"{(\Psi,1_\mu)}"] & \mu
	\end{tikzcd} \]
	commutes. This can be reduced to the commutativity of both 
	\[
		 \Grothendieck{\mu}{f} \circ \Grothendieck{\lambda,\mu}{\eta} \circ \Phi = \Psi \circ f \quad
		\hbox{and}\quad
		\big(\eta \bullet( \pr_1^{\lambda}\circ \Phi)\big) \circ 1_\lambda  = (1_\mu \bullet  f)\circ \eta.
	\]
	Again the second equality is immediate (both sides equate to \( \eta \)), so only the first equality has to be checked. 
	Thus the condition on \( (f,\eta) \) reduces to the commutativity of 
	\[ \begin{tikzcd}
		c\ar[dd,"1_c"'] \ar[dr,"\Phi"] \ar[rr,"1_c"]&&[2em] c \ar[dd,"1_c"'{near end}] \ar[dr,"\eta \circ \Phi"] \ar[rr,"f"]
							  &&d  \ar[dd,"1_d"{swap,near end}] \ar[dr,"\Psi"] 
\\
							  & \Grothendieck{c}{\lambda} \ar[dl,"\pr_1^{\lambda}"]
							  \ar[rr,"\Grothendieck{\lambda,\mu \circ f}{\eta}"{near start},crossing over]
							  &&\Grothendieck{c}{\mu \circ f} \ar[dl,"\pr_1^{\mu\circ f}"]
							  \ar[rr,"\Grothendieck{\mu}{f}"{near start}, crossing over]
							  && \Grothendieck{d}{\mu} \ar[dl,"\pr_1^{\mu}"]
		\\  c \ar[rr,"1_c"'] && c \ar[rr,"f"'] && d.
	\end{tikzcd} \]
	The only nontrivial condition of this is \(\Grothendieck{\mu}{f} \circ \eta \circ \Phi = \Psi \circ f \).
	However, this condition is equivalent to \( \eta \circ \Phi = \Psi \circ f \).
\end{itemize} 
It is thus immediate that the objects of \( \CoAlg(K) \) are the canonical dependent arrows of a (2-fam,$\Sigma$)-category introduced in \cite{ehrhardt2depCategories2024} and the arrows are obtained as combinations of lifts of \( (f,\eta) \colon \lambda \to \mu \) such that \( \eta \circ \Phi = [\Psi](f) \).

Conversely, if we first compute the canonical dependent arrow structure for a (2-fam,$\Sigma$)-category and then the higher comprehension category obtained from the 2-equivalence consists of the following data:
\begin{itemize}
	\item The comprehension category structure \(P \colon \C E \to \C B^{\to} \) consists of (2-)family arrow and \( \C C \)
	\item The category \( \C G  \) has as objects coalgebras \( c \to \Grothendieck{c}{\lambda} \) for the weakening comonad \( K \) described earlier. 
		Morphisms \( \Phi \to \Psi \)---where \( q(\Phi)= c \) and \( p(\Psi) = c' \)---are given by \( (f,\eta) \colon \lambda \to \mu \) such that \( \eta \circ \Phi = [\Psi](f) \).
\end{itemize} 
From the discussion above we thus infer that \( \C G = \CoAlg(K) \).

Comparing the functors, we compute that \( \Sigma \) takes \( \Phi\colon c \to \Grothendieck{c}{\lambda} \) to \( \lambda \) and \( (f,\eta) \) to \( (f,\eta) \), hence \( q = \Sigma \).
The functor \( \Delta \) takes \( \lambda \) to \( Q \) from \( \nu_\lambda = (Q, 1_{\lambda \circ \pr_1^\lambda \circ \pr_1^{\lambda \circ \pr_1^\lambda}}) \) and \( (f,\eta) \) to \( K(f,\eta) \) defined in \eqref{genCatFam::eq1}.
Examining how the second projection arrow is defined in the canonical way in \cite{petrakisCategoriesDependentArrows2023} we see that \( \pr_2^\lambda = Q \).
A similar computation shows that also \( \pr_2(f,\eta) = K(f,\eta) \).
Thus \( \pr_2(\lambda) = Q \).

The adjunction 
\[ \begin{tikzcd}
	\CoAlg(K) \ar[r,shift right = .8em,"\Sigma"{name = U,swap}]
	& \C E. \ar[l,shift right = .8em, "\Delta"{name = V, swap}]
\ar[from = U, to = V, "\dashv"{sloped,description}, phantom]
\end{tikzcd} \]
is not part of the higher comprehension category, but the natural transformation \( \chi\colon \lambda \circ \pr_1^\lambda \to \lambda\) can be extended to obtain this adjunction, as \( \chi \) and the counit coincide.

Thus we obtain that the diagram~\eqref{gCwFhCCequiv} can be extended to
\[ \begin{tikzcd}
	\depSC \ar[r,"\text{\footnotesize 2-$\simeq$}"{description},phantom] \ar[r,shift right = 1.3ex]
 &\HCompCat \ar[from = drr, bend right = 20] \ar[l,shift right = 1.3ex]
 \\
 \Sigma-\mathsf{IC} \ar[r,shift left = 1.3ex] \ar[u] \ar[r,"\text{\footnotesize 2-$\simeq$}"{description},phantom] 
	& \CompCat \ar[l,shift left = 1.3ex] \ar[r,"\text{\footnotesize bi-$\simeq$}"{description}, phantom] \ar[r,shift right = 1.3ex] \ar[u]
	& \WCmd \ar[l,shift right = 1.3ex] \ar[r,"\text{\footnotesize bi-$\simeq$}",description, phantom] \ar[r,shift right = 1.3ex] 
	& \gCwF. \ar[l,shift right = 1.3ex]
\end{tikzcd}
\]
The biequivalence \( \CompCat \simeq \gCwF \)  proved in \cite{coraglia2categoricalAnalysisContext2024} actually yields more than \cite[Theorem~3.5.8]{ehrhardt2depCategories2024}: not only does it give the canonical (dep,$\Sigma$)-structure given a (2-fam,$\Sigma$)-category, but also says that if we are given a (dep,$\Sigma$)-category, such that in the higher comprehension category, formed as in Lemma~\ref{2depSigmatohCC}, the natural transformation \( \chi \colon q \circ \pr_2 \To \id_{\C E} \) can be extended to an adjunction \( q \dashv \pr_2 \), then we obtain a pseudo-natural isomorphism \( \C G \to \CoAlg(K) \), where \( K \) is the comonad defined by bullet points \ref{weakCmd::cd1}--\ref{weakCmd::cd3} on page \pageref{weakCmd::cd1}.

However, not every (dep,$\Sigma$)-structure must be a canonical one. 
Observe that the (dep,$\Sigma$)-structure on a category obtained from a ring as in Examples~\hyperref[exm::ring1]{\ref{exm::typecategories}.\ref*{exm::ring1}}, \hyperref[exm::ring2]{\ref*{exm::indexedcategories}.\ref*{exm::ring2}} and~\ref{exm::ring3} is not the canonical one, as then the dependent arrows would be the arrows \(y \colon \ast \to \ast\) such that \(y + a \cdot b = 0  \), that is \( y = -a \cdot b \).
But as generally \( -a \cdot b \neq a - b \), these dependent arrows are not the same.

\chapter{Conclusions and future tasks}

In this paper we showed that the notion of a dependent structure on an indexed category with Sigma-objects generalises the notion of generalised categories with families, and that there are examples of such categories that can not be recovered within the latter approach. 
Thus, translating this back into the original notion of Petrakis and Ehrhardt, we proved the claim of the title, (2-dep,$\Sigma$)-categories are distinct from generalised categories with families and more general than these. 
The following tasks fall out of the scope of this paper, but we hope to undertake them in future work.
\begin{enumerate}
	\renewcommand{\labelenumi}{\textbf{\theenumi}.}
	\item
		It is possible to dualise the notion of a comprehension category, by requiring that in 
		\[ \begin{tikzcd}
			\C E \ar[dr,"q"'] \ar[rr,"Q"] && \C B^\to \ar[dl,"\dom"] 
			\\
						     &\C B
		\end{tikzcd} \]
		 \( q \) is an opfibration, and that \( Q \) sends \( q \)-cocartesian morphism to pushout squares in \( \C B^\to\).
		 This notion is mentioned in the introduction of~\cite{melliesComprehensionQuotientStructures2020}, but not elaborated on in this paper.
		 In~\cite{fumexInductionCoinductionSchemes2012} quotients for a fibration \( p \colon \C E \to \C B \) with a section \( \star \colon \C B \to \C E \) are defined as a left adjoint \( Q \dashv \star \).
		 From the unit \( \epsilon \colon \id_\C E \To \star \circ Q\) the authors then define a functor \( \pi \colon \C E \to \C B^\to \) by \( \pi(e) = p(\epsilon_e) \). 
		 This functor defines a commutative triangle
			\[ \begin{tikzcd}
				\C E \ar[dr,"p"'] \ar[rr,"\pi"] && \C B^\to \ar[dl,"\dom"] 
				\\
							     &\C B,
			\end{tikzcd} \]
		but \( p \) is a fibration. 
		Conversely, they define \emph{tC-opfibrations}, but these yield commutative squares
		\[ \begin{tikzcd}
			\C E \ar[dr,"q"'] \ar[rr,"Q"] && \C B^\to \ar[dl,"\codom"] 
			\\
						     &\C B,
		\end{tikzcd} \]
		where \( q \) is an opfibration. 
		Hence, it remains to find interesting examples of these ``co''comprehension categories, and check whether they can be extended to examples of higher ``co''comprehension categories.
	\item 
		Given a fibration \( p \colon \C E \to \C B \) and an opfibration \( q \colon \C F \to \C B \), one can form the pullback 
		\begin{equation} \begin{tikzcd}[execute at end picture={
			\rueckzug{1-1}{1-2}{2-1}{};
		}]
			p^*(\C F) \ar[d,"q^*p"'] \ar[r,"p^*q"] & \C E \ar[d,"p"]
			\\
			\C F \ar[r,"q"] & \C B,
		\end{tikzcd}
	\label{Concl::diag1}\end{equation}
		to obtain a ``combined'' structure on \( \C B \), through the diagonal.
		Alternatively, one can first form the corresponding functors \( F_p \colon \C B\opp \to \Cat \), \( F_q \colon \C B \to \Cat \), and then consider the functor \( \hat F \colon \C B\opp \times \C B \to \Cat \), defined as \( \hat F(b,b') = F_p(b) \times F_q(b') \).
		These approaches are not obviously connected, exemplified by the following diagram
		\[ \begin{tikzcd}
			\Fun(\C B\opp,\Cat) \times \Fun(\C B,\Cat) \ar[d] \ar[r] & \Fib_{\spl}(\C B) \times \OpFib_{\spl}(\C B) \ar[d]
			\\
			\Fun(\C B\opp \times \C B,\Cat) \ar[r] & ?
		\end{tikzcd} \]
		being \emph{not} commutative. 
		However, a connection would allow to transfer the obvious benefits of one approach to the other:
		\begin{itemize}
			\item 
				The first approach allows one to immediately ``attach'' the date of a comprehension category (and its dual), so the only question is what properties the diagonal in~\eqref{Concl::diag1} exhibits.
			\item 
				The latter approach does not allow to immediately transfer the definition of Sigma- and quotient-objects to the new functor, but instead gives the starting point for the following generalisation: 
				The category \( \C B\opp \times \C B \) is obviously self-dual, so it remains to see whether one can replace it with, for example a \( \ast \)-autonomous category (see~\cite{barrastAutonomousCategories1979,barrastAutonomousCategoriesOnce1999}--.which can be obtained from the Chu-construction \cite{chuConstructingastautonomousCategories}---and see what properties \( \hat F \) must have, to allow a decomposition into functors of the kind \( \C B\opp \to \Cat  \) and \( \C B \to \Cat \).
		\end{itemize}
		We hope to establish such a connection to unify the two approaches.
	\item 
		In~\cite{najmaeiSemanticsSyntaxType2025} a type theory for comprehension categories is developed, called CCTT, and is equipped with type formers for dependent functions.
		On the side of comprehension categories, these type formers are mirrored by dependent products, which are also introduced in this paper. 
		This definition of dependent products in a comprehension category differs from the dependent products on comprehension categories introduced by Jacobs in~\cite{jacobsComprehensionCategoriesSemantics1993}, albeit as the authors of~\cite{najmaeiSemanticsSyntaxType2025} note, they coincide if the comprehension category is full. 
		We wish to examine whether their type theory can be modified---e.g. by omitting some rules or relaxing prerequisites---or whether it needs to be extended appropriately, in order to obtain a type theory for higher comprehension categories.
\end{enumerate}

\small
\printbibliography

\end{document}